\documentclass{article} 
\usepackage{nips12submit_e,times}
\usepackage[T1]{fontenc}
\usepackage[latin9]{inputenc}
\usepackage{float}
\usepackage{amsthm}
\usepackage{amsmath}
\usepackage{amssymb}
\usepackage{subfigure}
\usepackage{graphicx}
\usepackage{appendix}

\makeatletter

\floatstyle{ruled}
\newfloat{algorithm}{tbp}{loa}
\providecommand{\algorithmname}{Algorithm}
\floatname{algorithm}{\protect\algorithmname}

\theoremstyle{plain}
\newtheorem{thm}{\protect\theoremname}
\theoremstyle{definition}
\newtheorem{defn}[thm]{\protect\definitionname}
\theoremstyle{plain}
\newtheorem{cor}[thm]{\protect\corollaryname}
\theoremstyle{remark}
\newtheorem{rem}[thm]{\protect\remarkname}
\newenvironment{lyxcode}
{\par\begin{list}{}{
\setlength{\rightmargin}{\leftmargin}
\setlength{\listparindent}{0pt}
\raggedright
\setlength{\itemsep}{0pt}
\setlength{\parsep}{0pt}
\normalfont\ttfamily}%
 \item[]}
{\end{list}}
\theoremstyle{plain}
\newtheorem{lem}[thm]{\protect\lemmaname}

\makeatother

\providecommand{\corollaryname}{Corollary}
\providecommand{\definitionname}{Definition}
\providecommand{\lemmaname}{Lemma}
\providecommand{\remarkname}{Remark}
\providecommand{\theoremname}{Theorem}

\title{Orthogonal Matching Pursuit with Noisy and Missing Data: Low and High Dimensional Results}

\author{
Yudong Chen \\
Department of Electrical and Computer Engineering \\
The University of Texas at Austin\\
Austin, TX 78712 \\
\texttt{ydchen@utexas.edu} \\
\And
Constantine Caramanis \\
Department of Electrical and Computer Engineering \\
The University of Texas at Austin\\
Austin, TX 78712 \\
\texttt{caramanis@mail.utexas.edu} \\
}

\nipsfinalcopy 

\begin{document}

\maketitle

\begin{abstract}
Many models for sparse regression typically assume that the covariates are known completely, and without noise. Particularly in high-dimensional applications, this is often not the case. This paper develops efficient OMP-like algorithms to deal with precisely this setting. Our algorithms are as efficient as OMP, and improve on the best-known results for missing and noisy data in regression, both in the high-dimensional setting where we seek to recover a sparse vector from only a few measurements, and in the classical low-dimensional setting where we recover an unstructured regressor. In the high-dimensional setting, our support-recovery algorithm {\it requires no knowledge} of even the statistics of the noise. Along the way, we also obtain improved performance guarantees for OMP for the standard sparse regression problem with Gaussian noise.
\end{abstract}

\section{Introduction}

Sparse Linear Regression, also popularly known as compressed sensing, deals with the problem of recovering a sparse vector from linear projections. These projections typically represent a physical measurement process, and as such, are often subject to noisy, missing or corrupted data. Standard algorithms, including popular approaches such as $\ell^1$-penalized regression, known as LASSO, are not equipped to deal with incomplete or noisy measurements. Not surprisingly, blindly running such algorithms on corrupted data gives solutions of significantly compromised quality, and indeed, we provide some computational results corroborating precisely this natural expectation. 

Recently, attention has turned to large-scale settings, where the data sets can be very large, and high-dimensional. Significantly, data collection in these settings can be even further prone to missing, noisy or corrupted data. In the high-dimensional setting, in particular, problems of prediction and inference critically hinge on correct identification of a low-dimensional structure -- sparsity, in the regression setting. Thus, algorithms that provably provide correct subset recovery are important, in addition to providing guarantees on $\ell^2$-error. Finally, the push towards large-scale learning calls for stable, simple and computationally efficient algorithms.

This paper focuses on precisely this problem. We present simple algorithms, no more complicated than the fastest algorithms run on clean (meaning, no additional noise, and no missing variables) covariates, whose statistical performance is as good as or better than any method known to us for noisy regression in the high dimensional setting, and in the low-dimensional setting. Indeed, the two main settings we consider are the regime where we have more observations than dimensions, but the signal (the regressor) exhibits no special structure such as sparsity (the ``classical'' or ``low-dimensional'' regime), and then the high-dimensional setting where dimensionality far outnumbers the number of measurements, but the signal is sparse (the standard compressed sensing setup). We describe our setting and assumptions in detail below, as well as provide a brief summary of recent work done in this area.

In the high-dimensional setting, our algorithm is  a greedy (OMP-like) algorithm. We provide conditions under which this algorithm is guaranteed (with high probability) to identify the correct support of the regressor. Interestingly, our support recovery algorithm requires no knowledge of the statistics of the corruption, in contrast with all other work we are aware of. Once the support has been identified, the problem reduces to one from the classical regime, since in the typical high dimensional scaling for compressed sensing, we have sparsity $k$ and $n = k \log p$ samples. This reduction to the low-dimensional setting turns out to be critical for both computational complexity, as well as stability and statistical performance. The central challenge in regression with noisy or otherwise corrupted covariates, $X$, is in estimating the covariance matrix $X^{\top}X$. This is exacerbated in the high-dimensional setting, where natural estimates may not even be positive semidefinite. In our setting, on the other hand, we estimate the support without requiring use of the covariance estimate, and hence avoid computational issues of non-convexity by moving directly to the low-dimensional setting. 

\subsection*{Related Work}
The problem of regression with noisy or missing covariates, in the high-dimensional regime, has recently attracted some attention, and several authors have considered this problem and made important contributions. Stadler and Buhlmann \cite{stadler2010missing} developed an EM algorithm to cope with missing data. While effective in practical examples, there does not seem to be a proof guaranteeing global convergence.  Recent
work has considered adapting existing approaches for
sparse regression with good theoretical properties to
handle noisy/missing data. The work by Rosenbaum and Tsybakov~\cite{rosenbaum2010sparse,rosenbaum2011improved} is among the first to obtain theoretical guarantees. They propose using a modified version of the Dantzig selector (they called it the MU selector) as follows. Letting $\mathbf{y} = X \beta + \mathbf{e}$, and $Z = X + W$ denote the noisy version of the covariates (we define the setup precisely, below), the standard Dantzig selector would minimize $\|\beta\|_1$ subject to the condition $\| Z^{\top}(\mathbf{y} - Z \beta)\|_{\infty} \le \tau$. Instead, they solve $\|  Z^{\top}(\mathbf{y} - Z \beta) + \mathbb{E}[W^{\top}W] \beta\|_{\infty} \le \mu\|\beta\|_1 + \tau$, thereby adjusting for the (expected) effect of the noise, $W$. Loh and Wainwright \cite{loh2011nonconvex} pursue a related approach, where they modify Lasso rather than the Dantzig selector. Rather than minimize $\|Z \beta - \mathbf{y}\|_2^2 + \lambda \|\beta\|_1$, they instead minimize a similarly adjusted objective: $\beta^{\top}(Z^{\top}Z - \mathbb{E}[W^{\top}W])\beta - 2\beta^{\top}Z^{\top}\mathbf{y} + \|\mathbf{y}\|_2^2 + \lambda \|\beta\|_1$. In this sense, their work is related to work by Xu and You~\cite{xu2007eiv} who consider a similar estimator but for noisy-regression in the low-dimensional setting. The modified Dantzig selector can be computed by solving a convex program. The modified Lasso formulation becomes non-convex. Interestingly, Loh and Wainwright show that the projected gradient descent algorithm finds a possibly local optimum that nevertheless has strong performance guarantees.

These methods obtain similar $\ell^2$-performance bounds, and recent work \cite{loh2012minimax} shows they are minimax optimal. Significantly, they both rely on knowledge of the covariance matrix $\mathbb{E}[W^{\top}W]$ of the  noise on the covariates.\footnote{The earlier work \cite{rosenbaum2010sparse} does not require that, but it does not guarantee support recovery, and its $ \ell_2 $ error bounds are weaker than the more recent approaches in~\cite{rosenbaum2011improved,loh2011nonconvex} that use $\mathbb{E}[W^{\top}W]$.} As our simulations demonstrate, this dependence seems critical: if the variance of the noise is either over- or under-estimated, the performance of the algorithms, even for support recovery, deteriorate considerably. The simple variant of the Orthogonal Matching Pursuit (OMP) algorithm we analyze requires no such knowledge for support recovery. Moreover, if $\mathbb{E}[W^{\top}W]$ is available, our algorithm has $\ell^2$-performance matching that in \cite{rosenbaum2011improved, loh2011nonconvex}.

Computationally, the methods mentioned above are more demanding than the simplest greedy methods that have proven theoretically and empirically successful in the clean-covariate case, i.e., when we have noiseless access to all the covariates. OMP~\cite{tropp2004greed,tropp2007OMP,davenport2010OMPRIP} is one of the greedy methods that have proven remarkably effective. These methods, however, have not been extended to the noisy or missing covariates case. Moreover, to the best of our knowledge, even the clean covariate case of sparse regression does not have a complete analysis for OMP under the noisy setting where measurements (response variables) are received with additive Gaussian noise. Many papers (e.g., \cite{donoho2006stable} or even more recent papers, e.g., \cite{JainTewariDhillonOMPR}) consider deterministic ($\ell^2$- or $\ell^{\infty}$-bounded) noise, and then obtain results for Gaussian noise as a corollary; however these results seem to be weaker than required. The work in \cite{FletcherRangan2009} considers the high-SNR case, so it is not clear that one can make use of the results there.

\subsection*{Contributions}
The present work develops a greedy OMP-like algorithm for the noisy case with noise: the setting where the measurements we receive have additive Gaussian noise, and moreover the version of the covariate (or sensing) matrix we get to see, has either additive noise or random erasures. Our algorithm is as efficient as standard OMP algorithms, and in the case of independent columns, our results are at least as good or better than any results available by any method known to us. While we conjecture that our results can be extended to the setting where the columns are not independent, and the sparsity is not known explicitly, we do not pursue this here. Specifically, the contributions of this paper are as follows:

\begin{enumerate}
\item Low-dimensional regime: For the case where the number of measurements, $n$, exceeds the dimensionality of the regressor, $k$, we design simple estimators for both the case of noisy covariates, and missing covariates. Our estimators are based on either knowledge of the statistics of the covariate noise, knowledge of the statistics of the covariate distribution, or knowledge of an Instrumental Variable correlated with the covariates. For both the case of missing and noisy data, we provide {\it finite sample} performance guarantees that are as far as we know, the best available.\footnote{Xu and You \cite{xu2007eiv} have shown asymptotic performance guarantees for some of the estimators we use, but have no finite sample results.} In the case of Instrumental Variables, we are not aware of any rigorous non-asymptotic  results.

Finally, we note that our results for the low-dimensional setting require no assumptions on the independence of columns of the covariate matrix.

\item High-dimensional regime: Next, we consider the standard high-dimensional scaling, when the regressor is $k$-sparse, but of dimension $p$, where $p$ greatly outnumbers the available measurements, $n$. We develop an iterative OMP-like algorithm. We give conditions for exact support recovery in the missing and noisy covariate setting, and provide $\ell^2$ error bounds for the regressor. For the case of independent columns, our results improve on  past results, to the best of our knowledge, both in terms of computational speed and performance. Interestingly, in the noisy $X$ setting, our support recovery algorithm {\it requires no knowledge of the statistics of the noise}; thus, our results imply that we have distributionally-robust support recovery. As far as we know (see also our simulations in Section \ref{sec:simulations}) other algorithms require some knowledge of the corruption statistics for support recovery.

\item In simulations, the advantage of our algorithm seems more pronounced, both in terms of speed, and in terms of statistical performance. Moreover, while we provide no analysis for the case of correlated columns of the covariate matrix, our simulations indicate that the impediment is in the analysis, as the results for our algorithm seem very promising. 

\item Finally, as a corollary to our results above, setting the covariate-noise-level to zero, we obtain bounds on the performance of OMP in the standard setting, with additive Gaussian measurement noise. Our bounds are better than bounds obtained by specializing deterministic results (for, e.g., $\ell^2$-bounded noise as in \cite{donoho2006stable}) and ignoring Gaussianity; meanwhile, while similar to the results in \cite{cai2011omp}, there seem to be gaps in their proof that we do not quite see how to fill.
\end{enumerate}

\subsection*{Paper Outline}
The outline of the paper is as follows. We first consider the low-dimensional or classical statistical setting in Section \ref{sec:low}. Here, the number of samples exceeds the dimensionality of the regressor. While important in its own right, this is critical for an OMP-based approach to sparse regression, since once the greedy approach has determined the (sparse) support set, the resulting problem is indeed a low-dimensional regression problem. Section \ref{sec:high} introduces the high-dimensional setting, and our OMP-like algorithm. Here we state the main results for this regime. Section \ref{sec:proofs} contains the proofs of our main results. Section \ref{sec:simulations} illustrates the performance and advantages of our algorithm empirically, and compares performance to other methods. 

\section{Problem Setup }

The main focus of this paper is on the high-dimensional setting. However, an important intermediate point is the consideration of the low-dimensional regime, since once the support of the sparse regressor has been identified, estimating those non-zero coefficients amounts to a low-dimensional problem. We thus define the setup in both settings.

We denote our unknown regressor (or signal) as $\beta^{*}$. In the low-dimensional setting, we have $\beta^{\ast} \in \mathbb{R}^k$, where as in the high dimensional setting, $\beta^{\ast} \in\mathbb{R}^{p}$ is an unknown $k$-sparse vector. For $i=1,\ldots,n$, we obtain measurement $y_{i}\in\mathbb{R}$ according to the linear model
$$
y_{i}=\left\langle \mathbf{x}_{i},\beta^{*}\right\rangle +e_{i}, \quad i=1,\dots,n.
$$
Here, $\mathbf{x}_{i}$ is the covariate vector of appropriate dimension ($\mathbf{x}_i \in \mathbb{R}^k$ for the low-dimensional case, and in $\mathbb{R}^p$ in the high dimensional case), and $e_{i}\in\mathbb{R}$ is additive error.

The standard setting assumes that each covariate vector $\mathbf{x}_i$ is known directly, and exactly. Instead, here we assume we only observe a vector
$\mathbf{z}_i \in \mathbb{R}^{k}$ (or $\mathbb{R}^p$) which is linked to $\mathbf{x}_i$ via some distribution. We focus on two cases:
\begin{enumerate}
\item Covariates with additive noise: We observe $\mathbf{z}_i=\mathbf{x}_i+\mathbf{w}_i$, where $\mathbf{w}_i \in \mathbb{R}^{k}$ (or $\mathbb{R}^p$) is the noise.

\item Covariates with missing data: We consider the case where the entries
of $\mathbf{x}_i$ are observed independently with probability $1-\rho$,
and missing with probability $\rho$. In particular, we assume the
following model: 
\[
\left(\mathbf{z}_i\right)_{j}=\begin{cases}
\left(\mathbf{x}_i\right)_{j} & w.p.\;(1-\rho)\\
0 & w.p.\;\rho
\end{cases},\quad\forall i,j.
\]

\end{enumerate}

Thus, in matrix notation, we have
$$
\mathbf{y} = X \beta^* + \mathbf{e},
$$
and we get to see: $(\mathbf{y},Z)$, where $Z = X + W$ in the noisy case, and $Z$ is the entry-erased version of $X$ in the missing data case. Thus, given $\{\mathbf{z}_i\}$ and $\{y_{i}\}$, we seek to estimate the unknown vector $\beta^{*}$. A task of particular importance in the high-dimensional setting is to estimate the support of $\beta^{\ast}$.
We use $\mathbf{x}_i$, $\mathbf{z}_i$ and $\mathbf{w}_i$ to denote the $i$th row of $X$, $Z$ and $W$,
respectively, and $X_i$, $Z_i$ and $W_i$ to denote the $i$th column.

In this paper, we consider the case where both the covariate matrix $X$ and the noise $ W $ and $\mathbf{e}$ are sub-Gaussian. We give the basic definitions here, as these are used throughout.
\begin{defn}
Sub-Gaussian Variable: A zero-mean variable $v$ is sub-Gaussian with parameter $\sigma_v > 0$ if for all $ \lambda \in \mathbb{R} $,
$$
\mathbb{E}[\exp(\lambda v)] \leq \exp(\sigma_v^2 \lambda^2 / 2).
$$
\end{defn}
\begin{defn}
Sub-Gaussian Matrix: A zero-mean matrix $V$ is called sub-Gaussian
with parameter $(\frac{1}{n}\Sigma,\frac{1}{n}\sigma^{2})$ if both of the following are
satisfied:
\begin{enumerate}
\item Each row $\mathbf{v}_{i}^{\top}\in\mathbb{R}^{p}$ of $V$ is sampled independently
and has $\mathbb{E}\left[\mathbf{v}_{i}\mathbf{v}_{i}^{\top}\right]=\frac{1}{n}\Sigma$.%
\footnote{The $\frac{1}{n}$ factor is used to simplify subsequent notation;
no generality is lost.
}
\item For any unit vector $\mathbf{u} \in \mathbb{R}^{p}$, $\mathbf{u}^{\top} \mathbf{v}_{i}$ is a sub-Gaussian
random variable with parameter at most $\frac{1}{\sqrt{n}}\sigma$.
\end{enumerate}
\end{defn}
Note that the second parameter $ \frac{1}{n} \sigma^2 $ is an upper bound.

In Section \ref{sec:low}, we need only a general sub-Gaussian assumption in order to guarantee our results. Thus we define:
\begin{defn}
Sub-Gaussian Design Model: We assume $X$, $W$ and $\mathbf{e}$ are sub-Gaussian
with parameters $(\frac{1}{n}\Sigma_{x},\frac{1}{n})$, $(\frac{1}{n}\Sigma_{w},\frac{1}{n}\sigma_w^2)$ and
$(\frac{1}{n}\sigma^{2},\frac{1}{n}\sigma_{e}^{2})$, respectively. We assume they are independent of each other.
\end{defn}

For our analytical results in the high-dimensional setting, we currently require independence across entries for $X$, $W$ and $\mathbf{e}$. Thus we define:
\begin{defn}
Independent Sub-Gaussian Design Model: We assume $X$, $W$ and $\mathbf{e}$ have zero-mean, independent and sub-Gaussian entries. The entries of $X$, $W$, and $\mathbf{e}$ have parameter $\frac{1}{n}$, $\frac{1}{n}\sigma_w^2$ and
$\frac{1}{n}\sigma_{e}^{2}$, respectively.
\end{defn}

Finally, we discuss our notation. In this paper, for the simplicity of exposition of the results but also of the analysis, we disregard constants that do not scale with $k$, $n$ or $p$ or any relevant variance $\sigma^2$. Thus, for example, writing $n \gtrsim f(k,p,\sigma)$ means  $ n\gtrsim cf(k,p,\sigma) $ for a positive universal constant $ c $ that does not scale with $k$, $n$, $p$, or $\sigma$.

The results we present in the sequel all hold with high probability (w.h.p.). By this, we mean with probability at least $1-C_{1}p^{-C_{2}}$, for positive constants $C_1$, $C_2$ independent of $n$, $p$, $k$ and $\sigma$ (i.e., all relevant variance quantities).

\section{The Low-Dimensional Problem}\label{sec:low}

We first consider the low-dimensional version of the problem where $\beta^{\ast} \in \mathbb{R}^k$, with
$k \ll n$. As noted above, in the high-dimensional sparse-regression setting, once we know the support of $\beta^{*}$, this is precisely the resulting problem. When $k \ll n$, the problem is strongly convex, and in the clean-covariate setting where we know $X$ exactly and completely, the solution is given by the standard least-square estimator:
\begin{eqnarray}
\hat{\beta} & = & (X^{\top}X)^{-1}X^{\top} \mathbf{y}=\arg\min_{\beta}\beta^{\top}(X^{\top}X)\beta-2 \mathbf{y}^{\top}X\beta. \label{eq:stdestimate}
\end{eqnarray}
In this setting, well-known results establish, among other measures of closeness to $\beta^{\ast}$, the following:
\begin{thm}[\cite{meinshausenyu2009lasso}] Suppose that (according to the sub-Gaussian design model defined above) $X$ is sub-Gaussian with parameters $(\frac{1}{n}\Sigma_x,\frac{1}{n})$, and the noise vector $\mathbf{e}$ is sub-Gaussian with parameters $(\frac{1}{n}\sigma^2,\frac{1}{n} \sigma_e^2)$. Moreover, suppose that $n \gtrsim \frac{k \log p}{\lambda^2_{\min}(\Sigma_x)}$. Then with high probability, the estimator above satisfies:
$$
\|\hat{\beta} - \beta^{\ast}\|_2 \lesssim \frac{\sigma_e}{\lambda_{\min}(\Sigma_x)} \sqrt{\frac{k \log p}{n}}.
$$
\end{thm}

The challenge in our setting is that we know only $Z$ (a noisy or partially deleted version of $X$), and hence cannot solve for $\hat{\beta}$. Some knowledge of $X$ or of the nature of the corruption ($W$ in the case of additive noise) is required in order to proceed. For the case of additive noise, we consider three models for {\it a priori} knowledge of $X$ or of the noise. For the case of partially missing data, we assume we know the erasure probability (easy to check directly). For additive noise, the models we use are as follows.
\begin{enumerate}
\item Noise Covariance: in this case, we assume we either know or somehow can estimate the noise covariance, $\Sigma_w = \mathbb{E} \left[W^{\top}W\right]$. We note that this is a typical assumption, e.g., \cite{loh2011nonconvex,xu2007eiv}. We also give results for when we can only {\it conservatively} estimate $\Sigma_w$. We are unaware of other such results. We discuss this further below.
\item Covariate Covariance: in this case, we assume that we either know or somehow can estimate the covariance of the true covariates, $\Sigma_{x}=\mathbb{E}\left[X^{\top}X\right]$. This assumption does not seem to be as common as the previous one, although it seems equally plausible to have an estimate of $\mathbb{E}\left[X^{\top}X\right]$ as of $\mathbb{E} \left[W^{\top}W\right]$.
\item Instrumental Variables: in this setting, we assume there are variables $U \in \mathbb{R}^{n\times m}$ with $m \ge k$, whose
rows are correlated with the rows of $X$, but independent of $W$ and $\mathbf{e}$, and that the realization of $U$ is known or can be estimated. Instrumental variables are common in the econometrics literature \cite{fuller1987measurementerror,carroll2006measurementerror}, and are often used when $X$ is not available. To the best of our knowledge, no rigorous non-asymptotic results are available when one has a noisy or partially erased version of the covariate matrix $X$.
\end{enumerate}
We note in this section, our results require no assumptions on the independence of the columns of $X$, $W$, or, therefore, of $Z$; that is, we assume we operate under the sub-Gaussian design model. As we discuss in more detail in Section \ref{sec:high}, our subset selection algorithm is iterative, and empirically works well for correlated or independent columns, although currently our analytical results (performance guarantees) do require this independence assumption, and hence the guarantees hold for the independent sub-Gaussian model.

Let us generically denote by $\hat{\Sigma}$ the estimator for $X^{\top}X$, and by $\hat{\gamma}$ our estimate for $X^{\top} \mathbf{y}$. The pair $(\hat{\Sigma},\hat{\gamma})$ depends on the assumption of what is known, i.e., according to the three possibilities outlined above. Thus, in place of $\hat{\beta} = (X^{\top}X)^{-1}X^{\top}\mathbf{y}$ given in (\ref{eq:stdestimate}), our proposed estimator for $\hat{\beta}$ naturally becomes:
\begin{eqnarray*}
\hat{\beta} & = & (\hat{\Sigma})^{-1}\hat{\gamma}=\arg\min_{\beta}\beta^{\top}(\hat{\Sigma})\beta-2\hat{\gamma}^{\top}\beta,
\end{eqnarray*}
where we require $ \hat{\Sigma}  $ to be positive semidefinite.
For this estimator, we have the following simple but general result. 
\begin{thm}
\label{prop:lowD_general} Suppose the following strong convexity condition holds:
\begin{eqnarray*}
\lambda_{\min}\left(\hat{\Sigma}\right) & \ge & \lambda>0.
\end{eqnarray*}
Then the estimation error satisfies 
\begin{eqnarray*}
\left\Vert \hat{\beta}-\beta^{*}\right\Vert _{2}
&\lesssim& \frac{1}{\lambda}\left\Vert \hat{\gamma}-\hat{\Sigma}\beta^{*}\right\Vert _{2}.
\end{eqnarray*}
\end{thm}
\begin{proof}
Let $\Delta=\hat{\beta}-\beta^{*}$. By optimality of $\hat{\beta}$,
we have $(\beta^{*}+\Delta){}^{\top}(\hat{\Sigma})(\beta^{*}+\Delta)-2\hat{\gamma}^{\top}(\beta^{*}+\Delta)\le\beta^{*\top}(\hat{\Sigma})\beta-2\hat{\gamma}^{\top}\beta^{*}.$
Rearranging terms gives $\Delta^{\top}\hat{\Sigma}\Delta\le2(\hat{\gamma}^{\top}-\beta^{*\top}\hat{\Sigma})\Delta$.
Under the strong convexity assumption, the l.h.s. is lower-bounded
by $\lambda\left\Vert \Delta\right\Vert _{2}^{2}$. The r.h.s.\ is upper-bounded
by $2\left\Vert \hat{\gamma}-\hat{\Sigma}\beta^{*}\right\Vert _{2}\left\Vert \Delta\right\Vert _{2}$
thanks to Cauchy-Schwarz. The result then follows. 
\end{proof}
This result is simple and generic. We specialize this bound to the case of additive noise, and missing variables, and in particular, in the case of additive noise we specialize it according to each estimator we form depending on the information available (as discussed above).

\subsection{Additive Noise}

As outlined above, we have three different approaches for obtaining information on $X^{\top}X$, and hence three different estimators, depending on what information we have available. Given this information, the resulting estimator is quite natural. We list these here, and subsequently provide the results obtained by specializing the main theorem above.
\begin{enumerate}
\item If $\Sigma_{w}$ is known, we use $\hat{\Sigma}=Z^{\top}Z-\Sigma_{w}$, $\hat{\gamma}=Z^{\top}y$. We note that this estimator has been previously studied in \cite{xu2007eiv}, but their analysis is asymptotic. Here we give finite-sample bounds.

\item If $\Sigma_{x}$ is known, we use $\hat{\Sigma}=\Sigma_{x}$, $\hat{\gamma}=Z^{\top}y$. This is simple, and as our results show, in certain regimes its performance improves that of $Z^{\top}Z-\Sigma_{w}$. While simple and natural, we were not able to find previous analysis with performance guarantees.
\item An Instrumental Variable (IV) $U\in\mathbb{R}^{n\times m}$ ($m\ge k$) is a matrix whose rows are correlated with the corresponding rows of $X$ but independent of $W$. If such an IV is known, we use $\hat{\Sigma}=Z^{\top}UU^{\top}Z$, $\hat{\gamma}=Z^{\top}UU^{\top}y$. While the use of instrumental variables is popular in the economics literature, we are unaware of previous analysis with non-asymptotic performance guarantees.
\end{enumerate}
\begin{rem} While assuming knowledge of $\Sigma_w$ is common, and indeed a central focus of this paper, in some applications it may only be reasonable to assume knowledge of an {\it upper bound} on the noise covariance. That is, we may only be able to obtain some estimate $\overline{\Sigma}_w$ such that $\overline{\Sigma}_w \succeq \Sigma_w$. Our algorithms and analysis carry over in this case, providing a somewhat weaker (as expected) guarantee for $\ell^2$ error.
\end{rem}

The results we present hold with high probability, where recall that by with high probability (w.h.p.) we mean with probability at least $1-C_{1}p^{-C_{2}}$, for positive constants $C_1$, $C_2$ independent of $n$, $p$, $k$ and $\sigma$. While in the reduction from the high-dimensional setting, the parameter $p$ has a natural interpretation as the original number of covariates (i.e., the dimension of $\beta^{\ast}$), here it is just a parameter that indexes the guarantees, trading off between accuracy and reliability.
\begin{cor}
[Knowledge of $\Sigma_{w}$]\label{cor:LD_add_1} Suppose $n\gtrsim\frac{(1+\sigma_{w}^{2})^{2}}{\lambda^2_{\min}(\Sigma_{x})}k\log p$ for $ p\ge k $.
Then, w.h.p., the estimator built using $\hat{\Sigma} =  Z^{\top}Z-\Sigma_{w}$ and $\hat{\gamma}=Z^{\top} \mathbf{y}$, satisfies
\[
\left\Vert \hat{\beta}-\beta^{*}\right\Vert _{2}\lesssim\frac{\left(\sigma_{w}+\sigma_{w}^{2}\right)\left\Vert \beta^{*}\right\Vert _{2}+\sigma_{e}\sqrt{1+\sigma_{w}^{2}}}{\lambda_{\min}(\Sigma_{x})}\sqrt{\frac{k\log p}{n}}.
\]
\end{cor}
\begin{rem}
Note that when $\sigma_{w}=0$, the bound reduces to the standard bound for the least-squares estimator; in particular, it implies exact recovery when $\sigma_{w}=\sigma_{e}=0$. 
Also, compared to existing results in \cite{xu2007eiv,loh2011nonconvex}, our bound makes clear the dependence on $\left\Vert \beta^{*}\right\Vert _{2}$. This is important and intuitive since $W$ is multiplied by $\beta^{\ast}$.
\end{rem}
\begin{rem} 
If we only have an upper bound, $\overline{\Sigma}_w \succeq \Sigma_w$, then using the same analysis one can show:
\[
\left\Vert \hat{\beta}-\beta^{*}\right\Vert _{2}\lesssim\frac{ \left[\left(\sigma_{w}+\sigma_{w}^{2}\right)\left\Vert \beta^{*}\right\Vert _{2}+\sigma_{e}\sqrt{1+\sigma_{w}^{2}} \right] \sqrt{\frac{k\log p}{n}} + \lambda_{\max}(\overline{\Sigma}_w - \Sigma_w) \|\beta^{\ast}\|_2}{\lambda_{\min}(\Sigma_{x}) - \lambda_{\max}(\overline{\Sigma}_w - \Sigma_w)}.
\]
with high probability as long as the denominator is positive.  Note that, as one might expect, the result is not consistent. Nevertheless, it allows us to quantify precisely the value of better estimation of the noise covariance.
\end{rem}

\begin{cor}[Knowledge of $\Sigma_{x}$]\label{cor:LD_add_2} Suppose $n\gtrsim  \log p$. Then, w.h.p.,
the estimator built using $\hat{\Sigma}=\Sigma_{x}$ and $\hat{\gamma}=Z^{\top} \mathbf{y}$, satisfies
\[
\left\Vert \hat{\beta}-\beta^{*}\right\Vert _{2}\lesssim\frac{\left(1+\sigma_{w}\right)\left\Vert \beta^{*}\right\Vert _{2}+\sigma_{e}\sqrt{1+\sigma_{w}^{2}}}{\lambda_{\min}(\Sigma_{x})}\sqrt{\frac{k\log p}{n}.}
\]
\end{cor}
\begin{rem}
(1) We only require $n\gtrsim \log p$ (the case where we use $\Sigma_{w}$ for our estimator requires the much more restrictive $n\gtrsim(1+\sigma_{w}^{2})^{2}k\log p$). The reason for this, is that here we don't estimate $\hat{\Sigma}$ from data, and hence do not require $\Omega(k \log p)$ samples in order to control $\lambda_{\min}(\hat{\Sigma})$ by $\lambda_{\min}(\Sigma_x)$, as is required in the previous result. (2) The bound is linear, rather than quadratic, in $\sigma_{w}$ (when $\sigma_{w}$ is large), but it does {\em not} vanish when $\sigma_{w}$ and $\sigma_{e}$ are zero. 
\end{rem}

\begin{rem}
The projected gradient method in Loh and Wainwright \cite{loh2011nonconvex} can be modified to use $\Sigma_{x}$ as the covariance estimator, and when we extend their algorithm in the natural way, a similar analysis yields the same error bound. 
\end{rem}

Suppose $\Sigma_{x}=I$. Comparing the above bounds:
\begin{eqnarray*}
\textrm{Knowledge of $\Sigma_w$:} &  & \left\Vert \hat{\beta}-\beta^{*}\right\Vert _{2}\lesssim\left[\left(\sigma_{w}+\sigma_{w}^{2}\right)\left\Vert \beta^{*}\right\Vert _{2}+\sigma_{e}\sqrt{1+\sigma_{w}^{2}}\right]\sqrt{\frac{k\log p}{n}.}\\
\textrm{Knowledge of $\Sigma_x$:} &  & \left\Vert \hat{\beta}-\beta^{*}\right\Vert _{2}\lesssim\left[\left(1+\sigma_{w}\right)\left\Vert \beta^{*}\right\Vert _{2}+\sigma_{e}\sqrt{1+\sigma_{w}^{2}}\right]\sqrt{\frac{k\log p}{n},}
\end{eqnarray*}
we see that the only difference is $\sigma_{w}^{2}$ vs. $1$. The $\sigma_{w}^{2}$
term arises from $(W^{\top}W-\Sigma_{w})$ while the $1$-term comes from $(X^{\top}X-\Sigma_{x})$.
The first bound is better when $\sigma_{w}^{2}<1$, and the other
way around when $\sigma_{w}^{2}>1$. This suggests the following strategy:
if we somehow know (or can estimate) both the variance of $X$ and $W$,
then we should use the first estimator if $\sigma_{w}^{2}<1$, otherwise use the second estimator. This gap in performance according to different regimes is in fact present, as confirmed in our simulations in Section \ref{sec:simulations}.

For the final result in this section, we use the following standard notation: For a matrix $A$, we let $\sigma_{i}(A)$ be the $i$-th singular value, so, e.g., $\sigma_1(A) = \sigma_{\max}(A)$.
\begin{cor}[Instrumental Variables]\label{cor:LD_add_3} Suppose the Instrumental Variable
$U\in\mathbb{R}^{n\times m}$ is zero-mean sub-Gaussian with parameter
$(\frac{1}{n}\Sigma_U,\frac{1}{n}\sigma_{u}^{2})$, and $\mathbb{E}\left[U^{\top}X\right]=\Sigma_{UX}.$
Let $\sigma_{1}=\sigma_{1}(\Sigma_{UX})$ and $\sigma_{k}=\sigma_{k}(\Sigma_{UX})$.
If $n\gtrsim\max\left\{ 1,\frac{\sigma_{u}^{2}(1+\sigma_{w}^{2})}{(m/k)\sigma_{k}^{2}}\right\} k\log p$,
then w.h.p. the estimator built using $\hat{\Sigma} = Z^{\top}UU^{\top}Z$, and $\hat{\gamma} = Z^{\top}UU^{\top}\mathbf{y}$, satisfies
\[
\left\Vert \hat{\beta}-\beta^{*}\right\Vert _{2}\lesssim\sqrt{\sigma_{w}^{2}\left\Vert \beta^{*}\right\Vert ^{2}+\sigma_{e}^{2}}\frac{\sigma_{1}\sigma_{u}}{\sigma_{k}^{2}\sqrt{k/m}}\sqrt{\frac{k\log p}{n}}.
\]
\end{cor}
\begin{rem}
Consider the term 
\[
\sqrt{\sigma_{w}^{2}\left\Vert \beta^{*}\right\Vert ^{2}+\sigma_{e}^{2}}\frac{\sigma_{1}\sigma_{u}}{\sigma_{k}^{2}\sqrt{k/m}}=\frac{\sqrt{\sigma_{w}^{2}\left\Vert \beta^{*}\right\Vert ^{2}+\sigma_{e}^{2}}}{(\sigma_{1}/\sigma_{u})\sqrt{\frac{k}{m}}}\cdot\frac{1}{\left(\sigma_{k}/\sigma_{1}\right)^{2}}.
\]
The first factor can be interpreted as 1/SNR, and the second is a measure of the correlation between $X$ and $U$ (i.e., the strength of the Instrumental Variable).
\end{rem}

\subsection{Missing Data}

As we do for noisy data, in the missing data setting we consider the sub-Gaussian Design model, so that $X$ is sub-Gaussian with possibly dependent columns. We assume that each entry of the covariate matrix $X$ is missing independently of all other entries, with probability $\rho$. Again, the key is in designing an estimator for $X^{\top}X$. We consider the natural choice, given our erasure model: We use $\hat{\Sigma}=(Z^{\top}Z)\odot M$ and $\hat{\gamma}=\frac{1}{(1-\rho)}Z^{\top}\mathbf{y}$, where $M_{ij}=\frac{1}{1-\rho}$ if $i=j$ or $\frac{1}{(1-\rho)^{2}}$ otherwise; here $ \odot  $ denotes element-wise product.
\begin{cor}
[Missing Data]\label{cor:LD_mis_1}Suppose $n\gtrsim\frac{1}{(1-\rho)^{4}\lambda_{\min}^{2}(\Sigma_{x})}k\log p$.
Then, w.h.p. our estimator satisfies
\[
\left\Vert \hat{\beta}-\beta^{*}\right\Vert _{2}\lesssim\left(\frac{1}{(1-\rho)^{2}}\left\Vert \beta^{*}\right\Vert _{2}+\frac{1}{1-\rho}\sigma_{e}\right)\frac{1}{\lambda_{\min}(\Sigma_{x})}\sqrt{\frac{k\log p}{n}}.
\]
\end{cor}
\begin{rem}
Note that as with the previous results, the dependence on $\left\Vert \beta^{*}\right\Vert _{2}$ is given explicitly.
\end{rem}

\section{The High-Dimensional Problem and Orthogonal Matching Pursuit}
\label{sec:high}

We now move to the high-dimensional setting. Given the results of the previous section, the main challenge that remains in the high-dimensional setting is to understand precisely when we can recover the correct support of $\beta^{\ast}$. With this accomplished, we can immediately apply the results of the low-dimensional setting. It is this which spares us from having to compute an estimate of $X^{\top}X$ in the high-dimensional setting, thus allowing us to avoid issues with non-convex optimization. 

The main contribution in this section is showing that our simple approach enjoys, as far as we know, the best known support recovery guarantees in the noisy and missing variable setting.

Our algorithm is iterative, and we would expect, as our experimental results in Section \ref{sec:simulations} corroborate, that it performs well despite correlation in the columns of $X$ and $W$. However, the line of analysis we develop in order to prove our performance guarantees, seems to require this independence. Hence, the results in this section all assume the {\it Independent} sub-Gaussian design model, i.e., the entries of $X$ and $W$ are assumed independent of each other and everything else.

Interestingly, this section shows that estimating the support of $\beta^{\ast}$, can be done {\it without using knowledge of $\Sigma_{w}$, $\Sigma_{x}$, or an instrumental variable}, but instead using $Z$ directly. A critical advantage to this approach is that while we pay a price in $ \ell_2 $ errors for not having an accurate estimate of $\Sigma_w$ (as in the remark of the previous section), our support recovery is {\it distributionally robust}, in the sense that our algorithm does not need to know $\Sigma_w$ or $\Sigma_x$ in order to recover the support (our guarantees, of course, are in terms of these quantities). Indeed, we use $Z$ directly for deciding on the next element to add to the support set. We note that this is not that surprising, given that the error is assumed to be unbiased (in particular, in dependent of $X$ and $\beta^{\ast}$) and standard OMP adds the element corresponding to largest correlation. Somewhat surprisingly, we also use $Z$ directly to estimate the residual, rather than use the estimators of the previous section. The advantage is that this allows distributionally-robust (i.e., distributionally oblivious) support recovery. Without this approach, we would have to control the propagation of the additional error obtained from conservative estimates of the noise. It seems that other algorithms in this space, e.g., \cite{loh2011nonconvex} require knowledge of $\Sigma_w$. We test the distributional robustness (i.e., the effect of not knowing the correct $\Sigma_w$) in our simulations results in Section \ref{sec:simulations}, where our results show that support recovery in \cite{loh2011nonconvex} deteriorates if $\Sigma_w$ is under- or over-estimated.

We consider the following modified OMP algorithm (Algorithm \ref{alg:1}). Note that the support recovery step is identical to OMP, where as in the final step recovering $\hat{\beta}$ is done using the estimators of the previous section. Given a matrix $Y$, we use $Y_{I}$ to denote the sub matrix with the columns of $Y$ with indices in $I$, and $Y_{I,I}$ the square submatrix of $Y$ with row and column indices in $I$. We note that unlike some of the latest results on OMP, we do not deal with stopping rules here, but instead simply assume we know the sparsity (or an upper bound of it). We believe that this can be relaxed by adapting the by-now standard stopping rules, although we do not pursue it here.

\begin{algorithm}[H]
\caption{mod-OMP\label{alg:1}}

\begin{lyxcode}
Input:~$Z$,$\mathbf{y}$,$k$

Initialize~$I=\phi$,~$I^{c}=\{1,2,\ldots,p\}$,~$\mathbf{r}=0$.

For~$j=1:k$

~~~~Compute~corrected~inner~products~$h_{i}=Z_{i}^{\top}\mathbf{r}$,~for~$i\in I^{c}$.

~~~~Let~$i^{*}=\arg\max_{i\in I^{c}}\left|h_{i}\right|$.

~~~~Set~$I=I\cup\{i^{*}\}$.

~~~~Update~residual~$\mathbf{r}=\mathbf{y}-Z_{I}(Z_{I}^{\top}Z_{I})^{-1}Z_{I}^{\top} \mathbf{y}$.

End

Set~$\hat{\beta}$~as:
\begin{eqnarray*}
\hat{\beta}_{I} & = & \hat{\Sigma}_{I,I}^{-1}\hat{\gamma}_{I}.\\
\hat{\beta}_{I^{c}} & = & 0,
\end{eqnarray*}
where $(\hat{\Sigma}_{I,I},\hat{\gamma}_I)$ are computed according to knowledge of $\Sigma_x$, $\Sigma_w$ or $U$.

Output~$\hat{\beta}$.~~~\end{lyxcode}
\end{algorithm}

\subsection{Guarantees for Additive Noise}
\begin{thm}
\label{thm:omp_add_support}Under the Independent sub-Gaussian Design model and Additive
Noise model, mod-OMP identifies the correct support of $\beta^{*}$ with high probability, provided
$n\gtrsim(1+\sigma_{w}^{2})^{2}k\log p$ and the non-zero entries
of $\beta^{*}$ is greater than 
\[
16\left(\sigma_{w}\left\Vert \beta^{*}\right\Vert _{2}+\sigma_{e}\right)\sqrt{\frac{\log p}{n}}.
\]
\end{thm}
\begin{rem}
(1) If $ k $ is an upper bound of the actual sparsity, then mod-OMP identifies a size-$ k $ superset of the support of $ \beta^* $. (2) Considering the clean-covariate case, $\sigma_w = 0$, we obtain results that seem to be stronger (better) than previous results for OMP with Gaussian noise and clean covariates. The work in \cite{cai2011omp} obtains a similar condition on the non-zeros of
$\beta^{*}$, however, their proof seems to implicitly require an independence between the residual columns and the noise vector which does not hold, and hence we are unable to complete the argument.\footnote{More specifically, in \cite{cai2011omp}, the proof of Theorem 8 applies the results of their Lemma 3 to bound $\|X^{\top} (I - X_I (X_I^{\top}X_I)^{-1}X_i^{\top}) \mathbf{e})\|_{\infty}$. As far as we can tell, however, Lemma 3 applies only to $\|X^{\top} \mathbf{e}\|_{\infty}$ thanks to independence, which need not hold for the case in question.}
\end{rem}
Once mod-OMP identifies the correct support, the problem reduces to a
low-dimensional one, which is exactly what we discuss in the previous
section. Thus, applying our low-dimensional results, we have the following
bounds on $\ell_{2}$ error, complementing the above results on support recovery.
\begin{cor}
\label{thm:omp_add} Consider the Independent sub-Gaussian model and Additive Noise model. If $n\gtrsim(1+\sigma_{w}^{2})^{2}k\log p$ and the non-zero entries of $\beta^{*}$ are greater than $\left(\sigma_{w}\left\Vert \beta^{*}\right\Vert _{2}+\sigma_{e}\right)\sqrt{1+\sigma_{w}^{2}}\sqrt{\frac{k\log p}{n}},$ then with high probability, the output of mod-OMP with estimator built from $(\hat{\Sigma},\hat{\gamma})$, satisfies:
\begin{enumerate}
\item (Knowledge of $\Sigma_w$): $\left\Vert \hat{\beta}-\beta^{*}\right\Vert _{2}\lesssim\left[\left(\sigma_{w}+\sigma_{w}^{2}\right)\left\Vert \beta^{*}\right\Vert _{2}+\sigma_{e}\sqrt{1+\sigma_{w}^{2}}\right]\sqrt{\frac{k\log p}{n}}.$
\item (Knowledge of $\Sigma_x$): \textup{$\left\Vert \hat{\beta}-\beta^{*}\right\Vert _{2}\lesssim\left[\left(1+\sigma_{w}\right)\left\Vert \beta^{*}\right\Vert _{2}+\sigma_{e}\sqrt{1+\sigma_{w}^{2}}\right]\sqrt{\frac{k\log p}{n}.}$ }
\item (Instrumental Variables): \textup{$\left\Vert \hat{\beta}-\beta^{*}\right\Vert _{2}\lesssim\left(\sigma_{w}\left\Vert \beta^{*}\right\Vert _{2}+\sigma_{e}\right)\frac{\sigma_{1}}{\sigma_{k}^{2}\sqrt{k/m}}\sqrt{\frac{k\log p}{n}}.$}
\end{enumerate}
\end{cor}

\subsection{Guarantees for Missing Data}
We provide a support-recovery result analogous to Theorem \ref{thm:omp_add_support} above.
\begin{thm}
\label{thm:omp_mis_support} Under the Independent sub-Gaussian Design model and missing data model, mod-OMP identifies the correct support of $\beta^{*}$ provided $n\gtrsim\frac{1}{(1-\rho)^{4}}k\log p$ and the non-zero entries of $\beta^{*}$ are greater than 
\[
\frac{16}{1-\rho}\left(\left\Vert \beta^{*}\right\Vert _{2}+\sigma_{e}\right)\sqrt{\frac{\log p}{n}}.
\]
\end{thm}
\begin{rem}
There is an unsatisfactory element to this bound, as it does not recover the clean-covariate results as $\rho \rightarrow 0$. Since the algorithm reduces to the standard OMP algorithm, this is a short falling of the analysis. At issue is a weakness in the concentration inequalities when $\rho$ is small. We note that the leading optimization-based bounds for performance under missing data, given in \cite{loh2011nonconvex}, seem to face the same problem. 
\end{rem}
We can combine the above theorem with Corollary \ref{cor:LD_mis_1} to obtain a bound on the $\ell_{2}$ error.
\begin{cor}
\label{thm:omp_mis} Consider the Independent sub-Gaussian model and Missing Data model. If $n \gtrsim\frac{1}{(1-\rho)^{4}}k\log p$ and the non-zero entries of $\beta^{*}$ are greater than $\frac{16}{1-\rho}\left(\left\Vert \beta^{*}\right\Vert _{2}+\sigma_{e}\right)\sqrt{\frac{\log p}{n}}$, then using our estimator for missing data, the output of mod-OMP satisfies
\[
\left\Vert \hat{\beta}-\beta^{*}\right\Vert _{2}\lesssim\left(\frac{1}{(1-\rho)^{2}}\left\Vert \beta^{*}\right\Vert _{2}+\frac{1}{1-\rho}\sigma_{e}\right)\sqrt{\frac{k\log p}{n}}.
\]

\end{cor}

\section{Proofs}
\label{sec:proofs}

We prove all the results in this section. We make continued use of a few technical lemmas on concentration results. We provide the statement of these lemmas here, but postpone the proofs to the Appendix.

\subsection{Supporting Concentration Results}
\begin{lem}
\cite[Lemma 14]{loh2011nonconvex}\label{lem:covariance_concentration} Suppose
$Y\in\mathbb{R}^{n\times k}$ is a zero mean sub-Gaussian matrix with
parameter $(\frac{1}{n}\Sigma,\frac{1}{n}\sigma^{2})$. If $n\gtrsim\log p\ge\log k$,
then
\[
\mathbb{P}\left(\left\Vert Y^{\top}Y-\Sigma\right\Vert _{\infty}\ge c_{0}\sigma^{2}\sqrt{\frac{\log p}{n}}\right)\le c_{1}\exp\left(-c_{2}\log p\right).
\]

\end{lem}

\begin{lem}\label{lem:Sigma_beta} Suppose $X\in\mathbb{R}^{n\times k}$, $Y\in\mathbb{R}^{n\times m}$
are zero-mean sub-Gaussian matrices with parameters $(\frac{1}{n}\Sigma_{x},\frac{1}{n}\sigma_{x}^{2})$,
$\left(\frac{1}{n}\Sigma_{y},\frac{1}{n}\sigma_{y}^{2}\right)$. Then for any fixed vectors $\mathbf{v}_1$, $\mathbf{v}_2$, we have
\[
\mathbb{P}\left(\left|\mathbf{v}_1^{\top}\left(Y^{\top}X-\mathbb{E}\left[Y^{\top}X\right]\right)\mathbf{v}_2\right|\ge t\left\Vert \mathbf{v}_1\right\Vert \left\Vert \mathbf{v}_2\right\Vert \right)\le3\exp\left(-cn\min\left\{ \frac{t^{2}}{\sigma_{x}^{2}\sigma_{y}^{2}},\frac{t}{\sigma_{x}\sigma_{y}}\right\} \right).
\]
In particular, if $n\gtrsim\log p \ge \log m \vee \log k$, we have w.h.p.
\[
\left|\mathbf{v}_1^{\top}\left(Y^{\top}X-\mathbb{E}\left[Y^{\top}X\right]\right)\mathbf{v}_2\right|\le\sigma_{x}\sigma_{y}\left\Vert \mathbf{v}_1\right\Vert \left\Vert \mathbf{v}_2\right\Vert \sqrt{\frac{\log p}{n}}.
\]
Setting $\mathbf{v}_1$ to be the $i^{th}$ standard basis vector, and using a union bound over $i=1,\ldots,m$,
we have w.h.p. \textup{
\[
\left\Vert \left(Y^{\top}X-\mathbb{E}\left[Y^{\top}X\right]\right)v\right\Vert _{\infty}\le\sigma_{x}\sigma_{y}\left\Vert v\right\Vert \sqrt{\frac{\log p}{n}}.
\]
}\end{lem}

As a simple corollary of this lemma, we get the following.
\begin{cor}
\label{cor:X_v_2norm} If $X\in\mathbb{R}^{n\times k}$ is a zero-mean
sub-Gaussian matrix with parameter $(\frac{1}{n}\Sigma_{x}^{2}I,\frac{1}{n}\sigma_{x}^{2})$,
and $\mathbf{v}$ is a fixed vector in $\mathbb{R}^{n}$, then for any $\epsilon\ge1$,
we have 
\[
\mathbb{P}\left(\left\Vert X^{\top}\mathbf{v}\right\Vert _{2}>\sqrt{\frac{(1+\epsilon)k}{n}}\sigma_{x}\left\Vert \mathbf{v}\right\Vert _{2}\right)\le3\exp\left(-ck\epsilon\right).
\]
\end{cor}

\begin{lem}
\label{lem:restricted_e} If $X\in\mathbb{R}^{n\times k}$, $Y\in\mathbb{R}^{n\times m}$
are zero mean sub-Gaussian matrices with parameter $(\frac{1}{n}\Sigma_{x},\frac{1}{n}\sigma_{x}^{2})$,$(\frac{1}{n}\Sigma_{y},\frac{1}{n}\sigma_{y}^{2})$,
then 
\[
\mathbb{P}\left(\sup_{\mathbf{v}_1\in\mathbb{R}^{m},\mathbf{v}_2\in\mathbb{R}^{k},\left\Vert \mathbf{v}_1\right\Vert =\left\Vert \mathbf{v}_2\right\Vert =1}\left|\mathbf{v}_1^{\top}\left(Y^{\top}X-\mathbb{E}\left[Y^{\top}X\right]\right)\mathbf{v}_2\right|\ge t\right)\le2\exp\left(-cn\min(\frac{t^{2}}{\sigma_{x}^{2}\sigma_{y}^{2}},\frac{t}{\sigma_{x}\sigma_{y}})+6(k+m)\right).
\]
In particular, for each $\lambda>0$, if $n\gtrsim\max\left\{ \frac{\sigma_{x}^{2}\sigma_{y}^{2}}{\lambda^{2}},1\right\} (k+m)\log p$,
then w.h.p.
\begin{eqnarray*}
\sup_{\mathbf{v}_1\in\mathbb{R}^{m},\mathbf{v}_2\in\mathbb{R}^{k}}\left|\mathbf{v}_1^{\top}\left(Y^{\top}X-\mathbb{E}\left[Y^{\top}X\right]\right)\mathbf{v}_2\right| & \le & \frac{1}{54}\lambda\left\Vert \mathbf{v}_1\right\Vert \left\Vert \mathbf{v}_2\right\Vert.
\end{eqnarray*}
\end{lem}

\subsection{Proof of Corollary \ref{cor:LD_add_1}}
\begin{proof}
Substituting $Z=X+W$ into the definition of $\hat{\gamma}$ and $\hat{\Sigma}$,
we obtain
\begin{eqnarray*}
\left\Vert \hat{\gamma}-\hat{\Sigma}\beta^{*}\right\Vert _{\infty} & = & \left\Vert -X^{\top}W\beta^{*}+Z^{\top}e-(W^{\top}W-\Sigma_{w})\beta^{*}\right\Vert _{\infty}\\
 & \le & \left\Vert X^{\top}W\beta^{*}\right\Vert _{\infty}+\left\Vert Z^{\top}e\right\Vert _{\infty}+\left\Vert (W^{\top}W-\Sigma_{w})\beta^{*}\right\Vert _{\infty}
\end{eqnarray*}
Using Lemma \ref{lem:Sigma_beta}, we have w.h.p.
\begin{eqnarray*}
\left\Vert X^{\top}W\beta^{*}\right\Vert _{\infty} & \le & \sigma_{w}\left\Vert \beta\right\Vert _{2}\sqrt{\frac{\log p}{n}}\\
\left\Vert Z^{\top}e\right\Vert _{\infty} & \le & \sigma_{e}\sqrt{1+\sigma_{w}^{2}}\sqrt{\frac{\log p}{n}}\\
\left\Vert (W^{\top}W-\Sigma_{w})\beta^{*}\right\Vert _{\infty} & \le & \sigma_{w}^{2}\left\Vert \beta\right\Vert _{2}\sqrt{\frac{\log p}{n}}.
\end{eqnarray*}
It follows that
\[
\left\Vert \hat{\gamma}-\hat{\Sigma}\beta^{*}\right\Vert _{2}\le\sqrt{k}\left\Vert \hat{\gamma}-\hat{\Sigma}\beta^{*}\right\Vert _{\infty}\le\left[\left(\sigma_{w}+\sigma_{w}^{2}\right)\left\Vert \beta^{*}\right\Vert _{2}+\sigma_{e}\sqrt{1+\sigma_{w}^{2}}\right]\sqrt{\frac{k\log p}{n}.}
\]

On the other hand, observe that $Z$ is sub-Gaussian with parameter
$(\frac{1}{n}\Sigma_{x}+\frac{1}{n}\Sigma_{w},\frac{1}{n}(1+\sigma_{w}^{2}))$.
When $n\gtrsim\frac{(1+\sigma_{w}^{2})^{2}k\log p}{\lambda^2_{\min}(\Sigma_{x})}$,
by Lemma \ref{lem:restricted_e} with $\lambda=\lambda_{\min}(\Sigma_{x})$,
we have $\lambda_{1}\left(Z^{\top}Z-(\Sigma_{x}+\Sigma_{w})\right)\le\frac{1}{54}\lambda_{\min}(\Sigma_{x})$ w.h.p.
It follows that
\begin{eqnarray*}
\lambda_{\min}\left(\hat{\Sigma}\right)=\inf_{\left\Vert v\right\Vert =1}v^{\top}\hat{\Sigma}v & = & \inf_{\left\Vert v\right\Vert =1}v^{\top}\left(\Sigma_{x}+Z^{\top}Z-(\Sigma_{x}+\Sigma_{w})\right)v\\
 & \ge & \lambda_{\min}(\Sigma_{x})-\lambda_{1}\left(Z^{\top}Z-(\Sigma_{x}+\Sigma_{w})\right)\\
 & \ge & \frac{1}{2}\lambda_{\min}(\Sigma_{x}).
\end{eqnarray*}
The proposition then follows by applying Theorem \ref{prop:lowD_general}.
\end{proof}

\subsection{Proof of Corollary \ref{cor:LD_add_2}}
\begin{proof}
In this case, we have
\begin{eqnarray*}
\left\Vert \hat{\gamma}-\hat{\Sigma}\beta^{*}\right\Vert _{\infty} & = & \left\Vert (X^{\top}X-\Sigma_{x})\beta^{*}+W^{\top}X\beta^{*}+Z^{\top} \mathbf{e} \right\Vert _{\infty}\\
 & \le & \left\Vert W^{\top}X\beta^{*}\right\Vert _{\infty}+\left\Vert Z^{\top} \mathbf{e} \right\Vert _{\infty}+\left\Vert (X^{\top}X-\Sigma_{x})\beta^{*}\right\Vert _{\infty}
\end{eqnarray*}
By Lemma \ref{lem:Sigma_beta}, we have w.h.p.
\begin{eqnarray*}
\left\Vert W^{\top}X\beta^{*}\right\Vert _{\infty} & \le & \sigma_{w}\left\Vert \beta\right\Vert _{2}\sqrt{\frac{\log p}{n}}\\
\left\Vert Z^{\top} \mathbf{e} \right\Vert _{\infty} & \le & \sigma_{e}\sqrt{1+\sigma_{w}^{2}}\sqrt{\frac{\log p}{n}}\\
\left\Vert (X^{\top}X-\Sigma_{x})\beta^{*}\right\Vert _{\infty} & \le & \left\Vert \beta^{*}\right\Vert _{2}\sqrt{\frac{\log p}{n}}.
\end{eqnarray*}
So $\left\Vert \hat{\gamma}-\hat{\Sigma}\beta^{*}\right\Vert _{2} \le \sqrt{k} \left\Vert \hat{\gamma}-\hat{\Sigma}\beta^{*}\right\Vert _{\infty}\lesssim\left[\left(1+\sigma_{w}\right)\left\Vert \beta^{*}\right\Vert _{2}+\sigma_{e}\sqrt{1+\sigma_{w}^{2}}\right]\sqrt{\frac{k\log p}{n}.}$
On the other hand, by assumption $\lambda_{\min}(\hat{\Sigma})=\lambda_{\min}(\Sigma_{x})$.
The proposition then follows by applying Theorem \ref{prop:lowD_general}.
\end{proof}

\subsection{Proof of Corollary \ref{cor:LD_add_3}}
\begin{proof}
First observe that
\begin{eqnarray*}
\lambda_{\min}(\hat{\Sigma}) & = & \lambda_{\min}((X+W){}^{\top}UU^{\top}(X+W))\\
 & = & \sigma_{k}^{2}(U^{\top}X+U^{\top}W)\\
 & = & \sigma_{k}^{2}(\mathbb{E}\left[U^{\top}X\right]+(U^{\top}X-\mathbb{E}\left[U^{\top}X\right])+U^{\top}W)\\
 & \ge & \left[\sigma_{k}-\sigma_{1}(U^{\top}X-\mathbb{E}\left[U^{\top}X\right])-\sigma_{1}(U^{\top}W)\right]^{2}.
\end{eqnarray*}
By Lemma \ref{lem:restricted_e} with $\lambda=\sigma_{k}$, we have
$\sigma_{1}(U^{\top}W)\le\frac{1}{4}\sigma_{k}$ and $\sigma_{1}(U^{\top}X-\mathbb{E}\left[U^{\top}X\right])\le\frac{1}{4}\sigma_{k}$
under our assumption, so $\lambda_{\min}(\hat{\Sigma})\ge\frac{1}{4}\sigma_{k}^{2}.$
On the other hand,
\begin{eqnarray*}
\left\Vert \hat{\Sigma}\beta^{*}-\hat{\gamma}\right\Vert _{2} & = & \left\Vert (X+W)^{\top}UU^{^{\top}}(W\beta^{*}-\mathbf{e})\right\Vert _{2}\\
 & \le & \left\Vert X^{\top}UU^{\top}(W\beta^{*}+\mathbf{e})\right\Vert _{2}+\left\Vert W^{\top}UU^{\top}(W\beta^{*}+ \mathbf{e})\right\Vert _{2}.
\end{eqnarray*}
We bound each term.
\begin{enumerate}
\item By Lemma \ref{lem:restricted_e} with $\lambda=\sigma_{1}$, we have
$\sigma_{1}(U^{\top}X)\le\frac{3}{2}\sigma_{1}$. Each entry of $W\beta^{*}+e$
is i.i.d. zero-mean sub-Gaussian with variance bounded by $\sigma_{w}^{2}\left\Vert \beta^{*}\right\Vert ^{2}+\sigma_{e}^{2}$.
Hence by Lemma \ref{lem:Sigma_beta}, $\left\Vert U^{\top}(W\beta^{*}+\mathbf{e})\right\Vert _{2}\le\sqrt{m}\left\Vert U^{\top}(W\beta^{*}+\mathbf{e})\right\Vert _{\infty}\le\sigma_{u}\sqrt{\sigma_{w}^{2}\left\Vert \beta^{*}\right\Vert ^{2}+\sigma_{e}^{2}}\sqrt{\frac{m\log p}{n}}$.
It follows that $\left\Vert \left(U^{\top}X\right)^{\top}U^{\top}(W\beta^{*}+\mathbf{e})\right\Vert _{2}\le2\sqrt{\sigma_{w}^{2}\left\Vert \beta^{*}\right\Vert ^{2}+\sigma_{e}^{2}}\sqrt{\frac{\sigma_{u}^{2}\sigma_{1}^{2}m\log p}{n}}$.
\item By Lemma \ref{lem:restricted_e} with $\lambda=\sigma_{1}$, we have
$\left\Vert W^{\top}U\right\Vert _{op}\le\sigma_{1}$ under the assumption,
so the second term is bounded by $\sigma_{w}\left\Vert \beta^{*}\right\Vert \sqrt{\frac{\sigma_{u}^{2}\sigma_{1}^{2}m\log p}{n}}$.
\end{enumerate}
The result follows from applying Theorem \ref{prop:lowD_general}.
\end{proof}

\subsection{Proof of Corollary \ref{cor:LD_mis_1}}
\begin{proof}
Let $\Sigma_{z}=\mathbb{E}\left[Z^{\top}Z\right]$; we have $(\Sigma_{z})_{ij}=(1-\rho)(\Sigma_{x})_{ij}$
for $i=j$ and $(\Sigma_{z})_{ij}=(1-\rho)^{2}(\Sigma_{x})_{ij}$
for $i\neq j$. Note that the observed matrix $Z$ is sub-Gaussian
with parameter $(\frac{1}{n}\Sigma_{z},\frac{1}{n})$, which follows
from the sub-Gaussianity of $X$ (c.f. \cite{loh2011nonconvex}).
We set $\Delta_{z}=Z^{\top}Z-\Sigma_{z}$. By Lemma \ref{lem:covariance_concentration},
we know $\max_{i}\left|(\Delta_{z})_{ii}\right|\le\frac{1}{4}(1-\rho)^{2}\lambda_{\min}(\Sigma_{x})$
w.h.p. When this happens, for each unit norm $v$, we have
\begin{eqnarray*}
v^{\top}(\Delta_{z}\odot M)v & = & \sum_{i,j}v_{i}v_{j}(\Delta_{z})_{ij}\frac{1}{(1-\rho)^{2}}+\sum_{i}v_{i}^{2}(\Delta_{z})_{ii}\left(\frac{1}{1-\rho}-\frac{1}{(1-\rho)^{2}}\right)\\
 & \le & \frac{1}{(1-\rho)^{2}}v^{\top}\Delta_{z}v+\left\Vert v\right\Vert ^{2}\frac{\rho}{(1-\rho)^{2}}\max_{i}\left|(\Delta_{z})_{ii}\right|\\
 & \le & \frac{1}{(1-\rho)^{2}}v^{\top}\Delta_{z}v+\frac{\rho}{4}\lambda_{\min}(\Sigma_{x}).
\end{eqnarray*}
By Lemma \ref{lem:restricted_e} with $\lambda=\frac{1}{4}(1-\rho)^{2}\lambda_{\min}(\Sigma_{x})$,
we obtain $\max_{v:\left\Vert v\right\Vert =1}v^{\top}\Delta_{z}v\le\frac{1}{4}(1-\rho)^{2}\lambda_{\min}(\Sigma_{x})$,
so $v^{\top}(\Delta_{z}\odot M)v\le\frac{1}{4}(1+\rho)\lambda_{\min}(\Sigma_{x})$.
Because $\hat{\Sigma}=(\Sigma_{z}+Z^{\top}Z-\Sigma_{z})\odot M=\Sigma_{x}+\Delta_{z}\odot M$,
it follows that $\lambda_{\min}(\hat{\Sigma})\ge\lambda_{\min}(\Sigma_{x})-\lambda_{1}\left(\Delta_{z}\odot M\right)\ge\frac{1}{2}\lambda_{\min}(\Sigma_{x})$.

On the other hand, observe that

\begin{eqnarray*}
\left\Vert \hat{\gamma}-\hat{\Sigma}\beta^{*}\right\Vert _{\infty} & \le & \left\Vert \hat{\gamma}-\Sigma_{x}\beta^{*}\right\Vert _{\infty}+\left\Vert (\hat{\Sigma}-\Sigma_{x})\beta^{*}\right\Vert _{\infty}\\
 & \le & \left\Vert \frac{1}{1-\rho}Z^{\top}X\beta^{*}-\Sigma_{x}\beta^{*}\right\Vert _{\infty}+\left\Vert \frac{1}{1-\rho}Z^{\top} \mathbf{e} \right\Vert _{\infty}+\left\Vert (\hat{\Sigma}-\Sigma_{x})\beta^{*}\right\Vert _{\infty}.
\end{eqnarray*}
By Lemma \ref{lem:Sigma_beta}, w.h.p. the first term is bounded by
$\frac{1}{1-\rho}\left\Vert \beta^{*}\right\Vert \sqrt{\frac{\log p}{n}}$,
and the second term is bounded by $\frac{1}{1-\rho}\sigma_{e}\sqrt{\frac{\log p}{n}}$.
The magnitude of the $i$-th term of $(\hat{\Sigma}-\Sigma_{x})\beta^{*}$
is
\begin{eqnarray*}
\left|((Z^{\top}Z-\mathbb{E}\left[Z^{\top}Z\right])_{i-}\odot M_{i-})\beta^{*}\right| &=& \left|(Z^{\top}Z-\mathbb{E}\left[Z^{\top}Z\right])_{i-}(M_{i-}^{\top}\odot\beta^{*})\right| \\
&\le& \left\Vert (Z^{\top}Z-\mathbb{E}\left[Z^{\top}Z\right])(M_{i-}^{\top}\odot\beta^{*})\right\Vert _{\infty}.
\end{eqnarray*}
Note that we use $M_{i-}$ to denote the $i^{th}$ row of the matrix $M$.

Thus, by Lemma \ref{lem:Sigma_beta} and union bound over $i$, we have
\begin{eqnarray*}
\left\Vert (\hat{\Sigma}-\Sigma_{x})\beta^{*}\right\Vert _{\infty} &\le& \max_{i=1,\ldots n}\left\Vert (Z^{\top}Z-\mathbb{E}\left[Z^{\top}Z\right])(M_{i-}^{\top}\odot\beta^{*})\right\Vert _{\infty} \\ 
&\le& \sqrt{\frac{\log p}{n}}\max_{i}\left\Vert M_{i-}^{\top}\odot\beta^{*}\right\Vert _{2} \\
&\le& \frac{1}{(1-\rho)^{2}}\left\Vert \beta^{*}\right\Vert _{\infty}\sqrt{\frac{\log p}{n}}.
\end{eqnarray*}
Combining pieces, we have 
\begin{eqnarray*} 
\left\Vert \hat{\gamma}-\hat{\Sigma}\beta^{*}\right\Vert _{2} &\le& \sqrt{k}\left\Vert \hat{\gamma}-\hat{\Sigma}\beta^{*}\right\Vert _{\infty} \\ 
&\le& \left(\frac{1}{(1-\rho)^{2}}\left\Vert \beta^{*}\right\Vert _{2}+\frac{1}{1-\rho}\sigma_{e}\right)\sqrt{\frac{k\log p}{n}}.
\end{eqnarray*}
The corollary follows by applying Theorem \ref{prop:lowD_general}.
\end{proof}

\subsection{Proof of Theorem \ref{thm:omp_add_support}}

We use induction. The inductive assumption is that the previous steps identify a
subset $I$ of the true support $I^{*}={\rm supp}(\beta^{*})$. Let $I_{r}=I^{*}-I$
be the remaining true support that is yet to be identified. We need
to prove that at the current step, mod-OMP picks an index in $I_{r}$,
i.e., $\left\Vert h_{I_{r}}\right\Vert _{\infty}>\left|h_{i}\right|$
for all $i\in(I^{*})^{c}$.

We use a decoupling argument similar to \cite{tropp2007OMP}: consider
the oracle which runs mod-OMP over only the true support $I^{*}$. Then
our mod-OMP identifies $I^{*}$ if and only if it identifies it in the
same order as the oracle. Therefore we can assume $I$ to be independent
of $X_{i}$ and $W_{i}$ for all $i\in(I^{*})^{c}$. Note that $I$
may still depend on $X_{I^{*}}$, $W_{I^{*}}$, and $\mathbf{e}$.

Define $\mathcal{P}_{I}\triangleq Z_{I}(Z_{I}^{\top}Z_{I})^{-1}Z_{I}^{\top}.$
We have
\begin{eqnarray}
 & & \left\Vert h_{I_{r}}\right\Vert _{\infty} \nonumber \\
 & = & \left\Vert Z_{I_{r}}^{\top}r\right\Vert _{\infty}\nonumber \\
 & = & \left\Vert Z_{I_{r}}^{\top}(I-\mathcal{P}_{I})(X_{I^{*}}\beta_{I^{*}}+\mathbf{e})\right\Vert _{\infty}\nonumber \\
 & = & \left\Vert Z_{I_{r}}^{\top}(I-\mathcal{P}_{I})(Z_{I^{*}}\beta_{I^{*}}^{*}-W_{I^{*}}\beta_{I^{*}}^{*}+\mathbf{e})\right\Vert _{\infty}\nonumber \\
 & = & \left\Vert Z_{I_{r}}^{\top}(I-\mathcal{P}_{I})(Z_{I_{r}}\beta_{I_{r}}^{*}-W_{I^{*}}\beta_{I^{*}}^{*}+\mathbf{e})\right\Vert _{\infty}\nonumber \\
 & \underset{=}{(a)} & \left\Vert X_{I_{r}}^{\top}(I-\mathcal{P}_{I})X_{I_{r}}\beta_{I_{r}}^{*}+W_{I_{r}}^{\top}(I-\mathcal{P}_{I})X_{I_{r}}\beta_{I_{r}}^{*}-Z_{I_{r}}^{\top}(I-\mathcal{P}_{I})W_{I}^{\top}\beta_{I}+Z_{I_{r}}^{\top}(I-\mathcal{P}_{I})\mathbf{e}\right\Vert _{\infty}\nonumber \\
\nonumber \\
 & \ge & \!\! \frac{1}{\sqrt{k-i}}\left(\left\Vert X_{I_{r}}^{\top}(I-\mathcal{P}_{I})X_{I_{r}}\beta_{I_{r}}^{*}\right\Vert _{2} \!-\! \left\Vert W_{I_{r}}^{\top}(I-\mathcal{P}_{I})X_{I_{r}}\beta_{I_{r}}^{*}\right\Vert _{2} \!-\! \left\Vert Z_{I_{r}}^{\top}(I-\mathcal{P}_{I})(W_{I}\beta_{I}-\mathbf{e})\right\Vert _{2}\right),\label{eq:h_good}
\end{eqnarray}
where (a) follows from substituting $ Z=X+W $. For the first term, we have the following lemma.
\begin{lem}
Under the assumptions of Theorem \ref{thm:omp_add_support}, w.h.p.
$\forall I_{1}\subseteq I^{*},$ $I_{1}^{c}\triangleq I^{*}-I_{1}$,
\begin{eqnarray*}
\lambda_{\min}\left(X_{I_{1}^{c}}^{\top}(I-\mathcal{P}_{I_{1}})X_{I_{1}^{c}}\right) & \ge & \frac{1}{2},\\
\lambda_{\max}\left(W_{I_{1}^{c}}^{\top}(I-\mathcal{P}_{I_{1}})X_{I_{1}^{c}}\right) & \le & \frac{1}{8}.
\end{eqnarray*}
\end{lem}
\begin{proof}
By Lemma \ref{lem:restricted_e} and a union bound, we have w.h.p.
$\forall I_{1}\subseteq I^{*}$, $\lambda_{\min}\left(X_{I_{1}^{c}}^{\top}X_{I_{1}^{c}}\right)\ge\frac{1}{2}$.
On the other hand, fixing $I_{1}\subseteq I^{*}$, we have
\begin{eqnarray*}
\left\Vert X_{I_{1}^{c}}^{\top}\mathcal{P}_{I_{1}}X_{I_{1}^{c}}\right\Vert _{op} & = & \left\Vert X_{I_{1}^{c}}^{\top}Z_{I_{1}}\left(Z_{I_{1}}^{\top}Z_{I_{1}}\right)^{-1}Z_{I_{1}}^{\top}X_{I_{1}^{c}}\right\Vert _{op}\\
 & \le & \sigma_{1}^{2}\left(X_{I_{1}^{c}}^{\top}Z_{I_{1}}\right)/\sigma_{\min}\left(Z_{I_{1}}^{\top}Z_{I_{1}}\right).
\end{eqnarray*}
Again by Lemma \ref{lem:restricted_e}, $\sigma_{\min}\left(Z_{I_{1}}^{\top}Z_{I_{1}}\right)\ge\frac{1}{2}(1+\sigma_{w}^{2})$
with probability at least $1-\exp\left(cn\frac{1}{(1+\sigma_{w}^{2})^{2}}+12k\right),$
and $\sigma_{1}^{2}\left(X_{I_{1}^{c}}^{\top}Z_{I_{1}}\right)\le\frac{1}{8}$
with probability at least $1-\exp\left(cn\frac{1}{(1+\sigma_{w}^{2})}+12k\right)$.
So a union bound over all $I_{1}$ yields w.h.p. $\forall I_{1}\subseteq I^{*}$,
$\left\Vert X_{I_{1}^{c}}^{\top}\mathcal{P}_{I_{1}}X_{I_{1}^{c}}\right\Vert _{op}\le\frac{1}{4}$.
It follows that 
$$
\lambda_{\min}\left(X_{I_{1}^{c}}^{\top}(I-\mathcal{P}_{I_{1}})X_{I_{1}^{c}}\right)\ge\lambda_{\min}\left(X_{I_{1}^{c}}^{\top}X_{I_{1}^{c}}\right)-\left\Vert X_{I_{1}^{c}}^{\top}\mathcal{P}_{I_{1}}X_{I_{1}^{c}}\right\Vert _{op}\ge\frac{1}{4}.
$$

Similarly, by Lemma \ref{lem:restricted_e} and the union bound, we have
w.h.p. $\forall I_{1}\subseteq I^{*}$, $\left\Vert W_{I_{1}^{c}}^{\top}X_{I_{1}^{c}}\right\Vert _{op}\le\frac{1}{16}$
and $\left\Vert W_{I_{1}^{c}}^{\top}\mathcal{P}_{I_{1}}X_{I_{1}^{c}}\right\Vert _{op}\le\frac{1}{16}$,
hence $\lambda_{\max}\left(W_{I_{1}^{c}}^{\top}(I-\mathcal{P}_{I_{1}})X_{I_{1}^{c}}\right)\le\left\Vert W_{I_{1}^{c}}^{\top}X_{I_{1}^{c}}\right\Vert _{op}+\left\Vert W_{I_{1}^{c}}^{\top}\mathcal{P}_{I_{1}}X_{I_{1}^{c}}\right\Vert _{op}\le\frac{1}{8}.$
\end{proof}
Therefore, the first term in (\ref{eq:h_good}) is lower bounded by $\frac{1}{4}\left\Vert \beta_{I_{r}}^{*}\right\Vert _{2}$,
and the second term is upper bounded by $\frac{1}{8}\left\Vert \beta_{I_{r}}^{*}\right\Vert _{2}$.

Now consider the third term in \ref{eq:h_good}. By Lemma \ref{lem:restricted_e}
and a union bound, we have w.h.p. $\sigma_{1}\left(W_{I_{1}}\right)\le\frac{3}{2}\sigma_{w}$
for all $I_{1}$. Lemma \ref{lem:Sigma_beta} gives $\left\Vert \mathbf{e} \right\Vert _{2} \le \frac{3}{2}\sigma_{e}$.
It follows that $\left\Vert (I-\mathcal{P}_{I_{1}})(W_{I_{1}}\beta_{I_{1}}-\mathbf{e})\right\Vert _{2}\le\sigma_{1}(I-P_{Z_{I_{1}}})\left(\sigma_{1}\left(W_{I_{1}}\right)\left\Vert \beta_{I_{1}}^{*}\right\Vert _{2}+\left\Vert \mathbf{e} \right\Vert _{2}\right)\le\frac{3}{2}\left(\sigma_{w}\left\Vert \beta_{I_{1}}^{*}\right\Vert _{2}+\sigma_{e}\right)$.
Set $v_{I_{1}}=(I-\mathcal{P}_{I})(W_{I}\beta_{I}-\mathbf{e})$. Because $Z_{I_{1}^{c}}$
and $v_{I_{1}}$ are independent, Corollary \ref{cor:X_v_2norm} gives
$\left\Vert Z_{I_{1}^{c}}^{\top}v_{I_{1}}\right\Vert _{2}\le\sqrt{\frac{(1+\epsilon)(k-i)(1+\sigma_{w}^{2})}{n}}\left\Vert v_{I_{1}}\right\Vert _{2}$
with probability at least $1-3\exp\left(ck\epsilon^{2}\right)$. Using
a union bound over all $I_{1}$, we conclude that the third term is
bounded w.h.p. by $4\sqrt{\frac{(1+\sigma_{w}^{2})(k-i)\log p}{n}}\left(\sigma_{w}\left\Vert \beta_{I}^{*}\right\Vert _{2}+\sigma_{e}\right).$

Combining the above bounds, we have 
$$
\left\Vert h_{I_{r}}\right\Vert _{\infty}\ge\frac{1}{\sqrt{k-i}}\left[\frac{1}{4}\left\Vert \beta_{I_{r}}^{*}\right\Vert _{2}-\frac{1}{8}\left\Vert \beta_{I_{r}}^{*}\right\Vert _{2}-4\sqrt{\frac{(1+\sigma_{w}^{2})(k-i)\log p}{n}}\left(\sigma_{w}\left\Vert \beta_{I}^{*}\right\Vert _{2}+\sigma_{e}\right)\right],
$$
which is greater than $\frac{1}{8\sqrt{k-i}}\left\Vert \beta_{I_{r}}^{*}\right\Vert _{2}$
if all the non-zero entries of $\beta^{*}$ are greater than $16\left(\sigma_{w}\left\Vert \beta^{*}\right\Vert _{2}+\sigma_{e}\right)\sqrt{\frac{(1+\sigma_{w}^{2})\log p}{n}}.$

On the other hand, by similar argument as above we have $\left\Vert (I-\mathcal{P}_{I})(Z_{I_{r}}\beta_{I_{r}}-W_{I^{*}}\beta_{I^{*}}+\mathbf{e})\right\Vert _{2}\le\frac{3}{2}\left(\left\Vert \beta_{I_{r}}^{*}\right\Vert _{2}+\sigma_{w}\left\Vert \beta_{I}^{*}\right\Vert _{2}+\sigma_{e}\right)$.
Note that for each $i\in I^{*c}$, $Z_{i}$ is independent of $Z_{I}$,
$X_{I^{*}}$ and $e$. Applying Corollary \ref{cor:X_v_2norm} gives
w.h.p.
\begin{eqnarray*}
\left|h_{i}\right| & = & \left|Z_{i}^{\top}(I-\mathcal{P}_{I})(X_{I^{*}}\beta_{I^{*}}+\mathbf{e})\right|\\
 & = & \left|Z_{i}^{\top}(I-\mathcal{P}_{I})(Z_{I_{r}}\beta_{I_{r}}-W_{I^{*}}\beta_{I^{*}}+\mathbf{e})\right|\\
 & \le & 4\sqrt{\frac{(1+\sigma_{w}^{2})\log p}{n}}\left(\left\Vert \beta_{I_{r}}^{*}\right\Vert _{2}+\sigma_{w}\left\Vert \beta_{I}^{*}\right\Vert _{2}+\sigma_{e}\right),
\end{eqnarray*}
which is smaller than $\frac{1}{8\sqrt{k-i}}\left\Vert \beta_{I_{r}}^{*}\right\Vert _{2}$
provided $n\gtrsim(1+\sigma_{w}^{2})^{2}k\log p$, and the nonzeros
of $\beta^{*}$ are greater than $4\left(\sigma_{w}\left\Vert \beta^{*}\right\Vert _{2}+\sigma_{e}\right)\sqrt{\frac{(1+\sigma_{w}^{2})\log p}{n}}$.
Using a union bound shows this holds for all $i\in I^{*c}$.

We conclude that $\left\Vert h_{I_{r}}\right\Vert _{\infty}>\left|h_{i}\right|$
for all $i\in I^{*c}$ w.h.p. This completes the proof.

\subsection{Proof of Theorem \ref{thm:omp_mis_support}}
\begin{proof}
Note that the entries of  $Z$ are i.i.d. sub-Gaussian random variables with parameter $\sqrt{\frac{1}{n}}$.
Similarly to the proof of Theorem \ref{thm:omp_add_support}, we use
induction, the decoupling argument, and the same notation. Therefore,
it suffices to show $\left\Vert h_{I_{r}}\right\Vert _{\infty}\ge\left|h_{i}\right|$
for all $i\in(I^{*})^{c}$.

We have
\begin{eqnarray*}
\left\Vert h_{I_{r}}\right\Vert _{\infty} & = & \left\Vert Z_{I_{r}}^{\top}(I-\mathcal{P}_{I})(X_{I^{*}}\beta_{I^{*}}+\mathbf{e})\right\Vert _{\infty}\\
 & \ge & \frac{1}{\sqrt{k-i}}\left\Vert Z_{I_{r}}^{\top}(I-\mathcal{P}_{I})\left(X_{I_{r}}\beta_{I_{r}}^{*}+(X_{I}-Z_{I})\beta_{I}^{*}+\mathbf{e}\right)\right\Vert _{2}\\
 & \ge & \frac{1}{\sqrt{k-i}}\left(\left\Vert Z_{I_{r}}^{\top}(I-\mathcal{P}_{I})X_{I_{r}}\beta_{I_{r}}^{*}\right\Vert _{2}+\left\Vert Z_{I_{r}}^{\top}(I-\mathcal{P}_{I})(X_{I}-Z_{I})\beta_{I}^{*}\right\Vert _{2}-\left\Vert Z_{I_{r}}^{\top}(I-\mathcal{P}_{I})\mathbf{e}\right\Vert _{2}\right)
\end{eqnarray*}
Consider the first term. We have $\lambda_{\min}(Z_{I_{r}}^{\top}X_{I_{r}})\ge\frac{1}{2}(1-\rho)$
by Lemma \ref{lem:restricted_e}. We also have $\lambda_{\min}(Z_{I}^{\top}Z_{I})\ge\frac{1}{2}(1-\rho)$,
$\sigma_{1}(Z_{I_{r}}^{\top}Z_{I})\le\frac{1}{8}(1-\rho)^{2}$, $\sigma_{1}(Z_{I}^{\top}X_{I_{r}})\le\frac{1}{8}(1-\rho)^{2}$
by the same lemma. It follows that $\lambda_{1}(Z_{I_{r}}^{\top}\mathcal{P}_{I}X_{I_{r}})=\lambda_{1}(Z_{I_{r}}^{\top}Z_{I}(Z_{I}^{\top}Z_{I})^{-1}Z_{I}^{\top}X_{I_{r}})\le\sigma_{1}^{2}(Z_{I}^{\top}Z_{I_{r}})/\lambda_{\min}(Z_{I}^{\top}Z_{I})\le\frac{1}{4}(1-\rho)^{3}$.
We conclude that $\lambda_{\min}(Z_{I_{r}}^{\top}(I-\mathcal{P}_{I})Z_{I_{r}})\ge\lambda_{\min}(Z_{I}^{\top}Z_{I})-\lambda_{1}(Z_{I_{r}}^{\top}\mathcal{P}_{I}Z_{I_{r}})\ge\frac{1}{4}(1-\rho)$.
So the first term is at least $\frac{1-\rho}{4\sqrt{k-i}}\left\Vert \beta_{I_{r}}^{*}\right\Vert _{2}$.

For the second term, we apply Lemma \ref{lem:restricted_e} to obtain that
w.h.p., $\sigma_{1}(X_{I}-Z_{I}) \le 2$. It follows that
\begin{eqnarray*}
\left\Vert (I-\mathcal{P}_{I})(X_{I}-Z_{I})\beta_{I}\right\Vert _{2} & \le & 2\left\Vert \beta_{I}\right\Vert _{2}.
\end{eqnarray*}
By Corollary \ref{cor:X_v_2norm} and a union bound, we obtain
\[
\left\Vert Z_{I_{r}}^{\top}(I-\mathcal{P}_{I})(X_{I}-Z_{I})\beta_{I}^{*}\right\Vert _{2}\le2\left\Vert \beta_{I}\right\Vert _{2}\sqrt{\frac{(k-i)\log p}{n}},
\]
which is smaller than $\frac{1-\rho}{8}(1-\rho)\left\Vert \beta_{I_{r}}^{*}\right\Vert _{2}$
if the non-zeros are bigger than $\frac{16}{1-\rho}\left\Vert \beta_{I}\right\Vert _{2}\sqrt{\frac{\log p}{n}}.$

Consider the third term. In the proof of Theorem \ref{thm:omp_add_support}
we have shown that $\left\Vert \mathbf{e} \right\Vert _{2}\le\sigma_{e}$, so
$\left\Vert (I-\mathcal{P}_{I})\mathbf{e} \right\Vert _{2}\le\sigma_{e}$. w.h.p.
By Corollary \ref{cor:X_v_2norm} and a union bound, it follows that
$\left\Vert Z_{I_{r}}^{\top}(I-\mathcal{P}_{I}) \mathbf{e} \right\Vert _{2}\le\sqrt{\frac{(k-i)\log p}{n}}\sigma_{e}$,
which is smaller than $\frac{1-\rho}{16}\left\Vert \beta_{I_{r}}^{*}\right\Vert _{2}$
if non-zeros are bigger than $\frac{16}{1-\rho}\sigma_{e}\sqrt{\frac{\log p}{n}}.$

Combining the above bounds, we conclude that $\left\Vert h_{I_{r}}\right\Vert _{\infty}\ge\frac{1-\rho}{8\sqrt{k-i}}$$\left\Vert \beta_{I_{r}}\right\Vert _{2}$
if all the non-zero entries of $\beta^{*}$ is greater than $\frac{16}{1-\rho}(\left\Vert \beta^{*}\right\Vert _{2}+\sigma_{e})\sqrt{\frac{\log p}{n}}$.

We now consider $\left|h_{i}\right|$ for $i\in(I^{*})^{c}$. We have
\begin{eqnarray*}
\left\Vert (I-\mathcal{P}_{I})(X_{I^{*}}\beta_{I^{*}}+\mathbf{e})\right\Vert _{2} & \le & \left\Vert X_{I^{*}}\beta_{I^{*}}+\mathbf{e}\right\Vert _{2}\\
 & \le & \frac{3}{2}\left\Vert \beta^{*}\right\Vert _{2}+\sigma_{e}.
\end{eqnarray*}
So by independence of $Z_{i}$ and $X_{I^{*}}$ and Corollary \ref{cor:X_v_2norm},
we obtain
\begin{eqnarray*}
\left|h_{i}\right| & = & \left|Z_{i}^{\top}(I-\mathcal{P}_{I})(X_{I^{*}}\beta_{I^{*}}+ \mathbf{e})\right|\\
 & \le & \sqrt{\frac{\log p}{n}}(\frac{3}{2}\left\Vert \beta^{*}\right\Vert _{2}+\sigma_{e}),
\end{eqnarray*}
which is smaller than $\frac{1-\rho}{8\sqrt{k-i}}\left\Vert \beta_{I_{r}}^{*}\right\Vert _{2}$
if all the non-zeros of $\beta^{*}$ are bigger than $\frac{16}{1-\rho}(\left\Vert \beta^{*}\right\Vert _{2}+\sigma_{e})\sqrt{\frac{\log p}{n}}$.
\end{proof} 

\section{Numerical Simulations}
\label{sec:simulations}
In this section we provide numerical simulations that corroborate the theoretical results presented above, as well as shed further light on the performance of mod-OMP for noisy and missing data. Our results illustrate, in particular, several key points. First, in both the low-dimensional and high-dimensional settings, empirical results demonstrate that the scaling promised in the corollaries to Theorem \ref{prop:lowD_general} and Theorem \ref{thm:omp_add_support} is correct. We demonstrate this by rescaling the error of our experiments, normalizing by the predicted contribution to the error of $n$, $k$ and $p$, in order to highlight the dependence on $\sigma_w$. Our experiments show a clear alignment of the actual results along the predicted results. The results of this section also show the different regimes of efficacy of our different estimators for the noisy-covariate setting. Finally, we also compare to \cite{loh2011nonconvex}, and demonstrate that in addition to faster running time, we seem to obtain better empirical results in terms of recovery errors.

We present the low-dimensional results first, and then the high-dimensional results.

\subsection{The Low-Dimensional Case}
We report some simulation results on our low-dimensional results from Section \ref{sec:low}. These results are also relevant to the high-dimension setting, as our OMP algorithm reduces a high-dimensional problem to a low-dimensional one once it identifies the correct support. Note that each of our bounds in Corollary \ref{cor:LD_add_1} to Corollary \ref{cor:LD_mis_1} scales with $\frac{\log p}{n}$, which is to be expected. Therefore, we focus on verifying the scaling with the other parameters such as $k, \sigma_w, \rho$ and $\|\beta^*\|$.

We first look at the case with additive noise. We fix $ n = 3200 $, $ \sigma_e=0 $, $\Sigma_x = I$ and $ \Sigma_w = \sigma^2_w I $, and sample all matrices from a Gaussian distribution. $k$ and $\sigma_w$ take values in ${2,3,\ldots,7} $ and $[0,2]$, respectively. For each $k$, we generate the true $ \beta^* $ as a random $\pm 1$ vector; note that $\|\beta^*\|=\sqrt{k}$, which also scales with $ k $. Figure \ref{fig:add_low_1} (a) shows the $\ell_2$ recovery error under different $k$ and $\sigma_w$ using the estimator built from knowledge of $\Sigma_w$, where one can see the quadratic dependence on $\sigma_w$. Corollary \ref{cor:LD_add_1} predicts that, with fixed $n$, the error scales proportional to $(\sigma_w + \sigma_w^2) \|\beta^*\| \sqrt{k\log p} = (\sigma_w + \sigma_w^2) k \sqrt{\log p}$; in particular, if we plot the error versus the control parameter $ (\sigma_w + \sigma_w^2) k $, all curves should roughly become straight lines through the origin and  align with each other. Indeed, this is precisely what we see; the results, representing results averaged over 100 trials, are plotted in Figure \ref{fig:add_low_1} (b).

Similarly, we performed simulations for the estimators built from knowledge of $\Sigma_x$ and from Instrumental Variables. In the latter case, the Instrumental Variable is generated by $U=X\Gamma + E$, where $ \Gamma\in\mathbb{R}^{k\times m}$ with $ m=2k $ and the entries of $\Gamma $ and $E$ being i.i.d. standard Gaussian variables; in this case we have $ \sigma_1(\Sigma_{UX}) \approx \sigma_k(\Sigma_{UX}) = \Theta(\sqrt{m})$ and $\sigma_u=\sqrt{k}$. Corollaries \ref{cor:LD_add_2} and \ref{cor:LD_add_3} predict that the $ \ell_2 $ errors are proportional to the control parameters $ (1+\sigma_w)k $ and $ \sigma_w k $.  respectively. These predictions again match well our simulation results shown in Figure \ref{fig:add_low_23} (a) and (b).

In addition, we compare the performance of the estimators built from $\Sigma_w$ and $\Sigma_x$. Figure \ref{fig:add_low_1v2} shows their recovery error under different $\sigma_w$ with $k=7$. The results match the theory, and in particular show that the scaling depends as predicted on $\sigma_w$: The $\Sigma_w$-estimator performs better for small $\sigma_w$, and in particular, delivers exact recovery when $\sigma_w = 0$; the $\Sigma_x$-estimator is more favorable for large $\sigma_w$ due to its linear dependence on $\sigma_w$ (versus quadratic), but the error does not go to zero when $\sigma_w \rightarrow 0$. The crossover occurs roughly at $\sigma_w=1$.

\begin{figure}
\centering{
\begin{tabular}{cc}
\includegraphics[scale=.7]{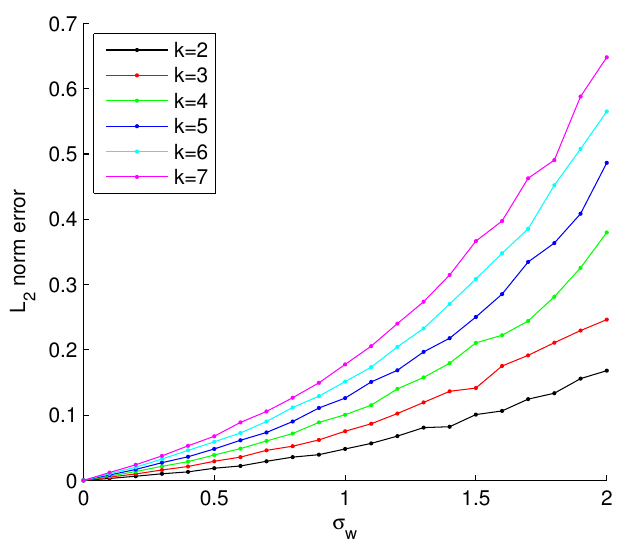} &
\includegraphics[scale=.7]{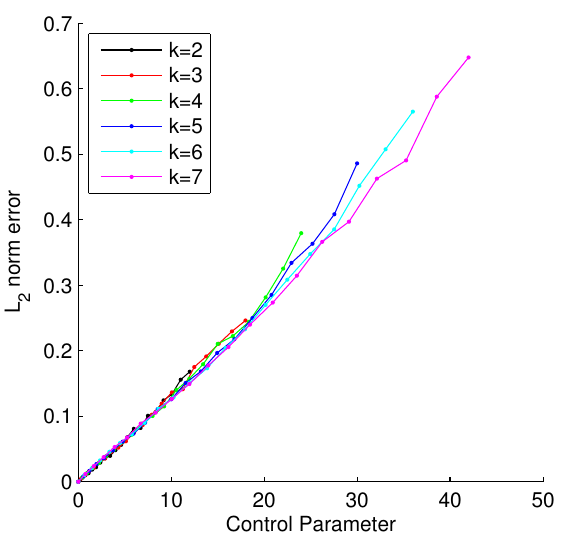} \\
(a) & (b)
\end{tabular}
\caption{\label{fig:add_low_1} $ \ell_2 $ recovery error of the $\Sigma_w$-estimator for additive noise versus (a) the noise magnitude $\sigma_w$ and (b) the control parameter $ (\sigma_w + \sigma_w^2) k $. As predicted by Corollary \ref{cor:LD_add_1}, all curves in (b) are roughly straight lines and align. Each point is an average over 200 trials.}
}
\end{figure}

\begin{figure}
\centering{
\begin{tabular}{cc}
\includegraphics[scale=.7]{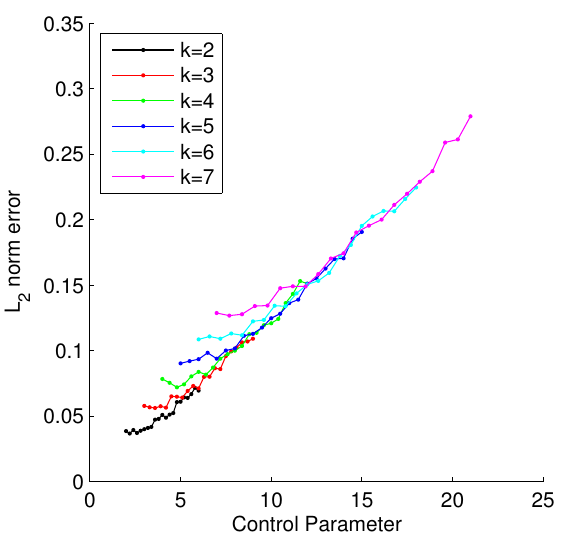} &
\includegraphics[scale=.7]{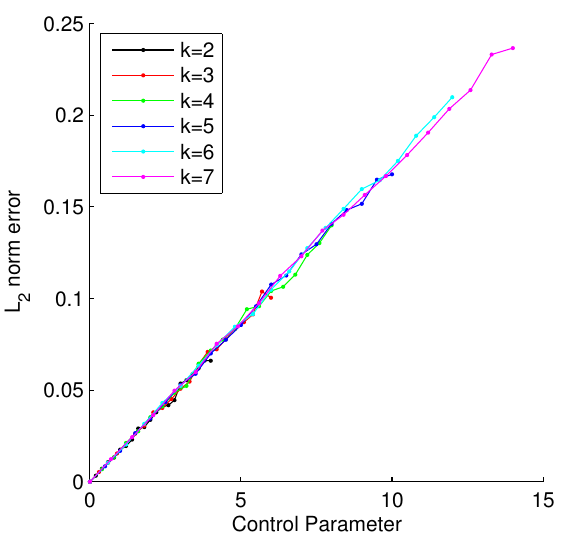} \\
(a) & (b)
\end{tabular}
\caption{\label{fig:add_low_23} (a) $ \ell_2 $ recovery error of the $\Sigma_x$-estimator for additive noise versus the control parameter $ (1+\sigma_w) k $. (b) $ \ell_2 $ recovery error of the IV-estimator versus the control parameter $ \sigma_w k $. As predicted by Corollary \ref{cor:LD_add_2} and \ref{cor:LD_add_3}, all curves are roughly straight lines and align. Each point is an average over 100 trials.}
}
\end{figure}

\begin{figure}
\centering{
\begin{tabular}{c}
\includegraphics[scale=.7]{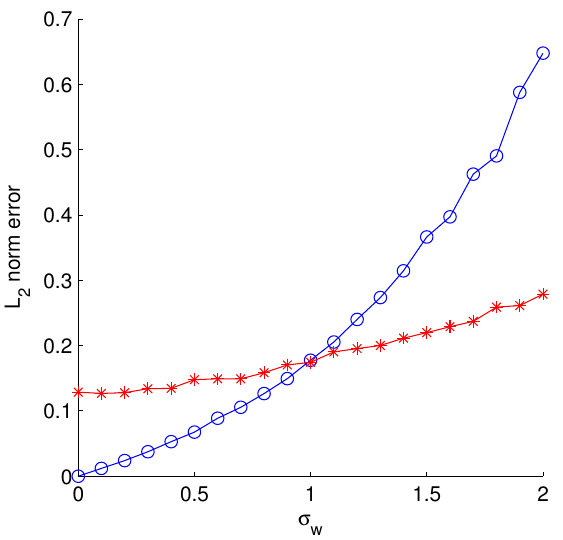}
\end{tabular}
\caption{\label{fig:add_low_1v2} Comparison between recovery errors of the $\Sigma_w$- and $\Sigma_x$-estimators for additive noise. Each point is an average over 100 trials.}
}
\end{figure}

Finally, we turn to the case with missing data. We perform simulations with parameters $n=2000$, $k\in\{2,\ldots,7\}$,  $ \rho\in[0,0.8]$, and $\beta^*$ generated in the same way as above (so that $\|\beta^*\|=k$). With $n$ fixed, Corollary \ref{cor:LD_mis_1} guarantees that the recovery error is bounded by $ O(\frac{k}{(1-\rho)^2}) $. The simulation results in Figure \ref{fig:mis_low_1} seem to \emph{outperform} this bound, as  the error goes to zero when $\rho \rightarrow 0$. If we plot the error versus the control parameter $ k \frac{\sqrt{\rho}}{1-\rho} $, then the curves become roughly straight lines and align. It would be interesting in the future to tighten our bound to match this scaling.

\begin{figure}
\hspace{-0.8cm}
\begin{tabular}{cc}
\includegraphics[scale=.65]{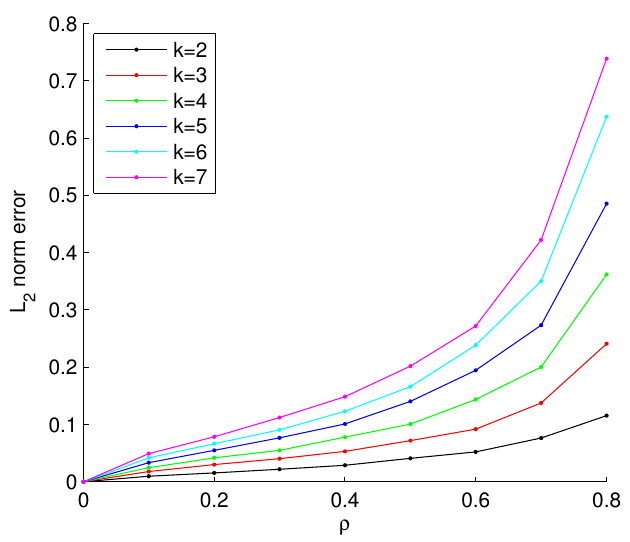} &
\includegraphics[scale=.65]{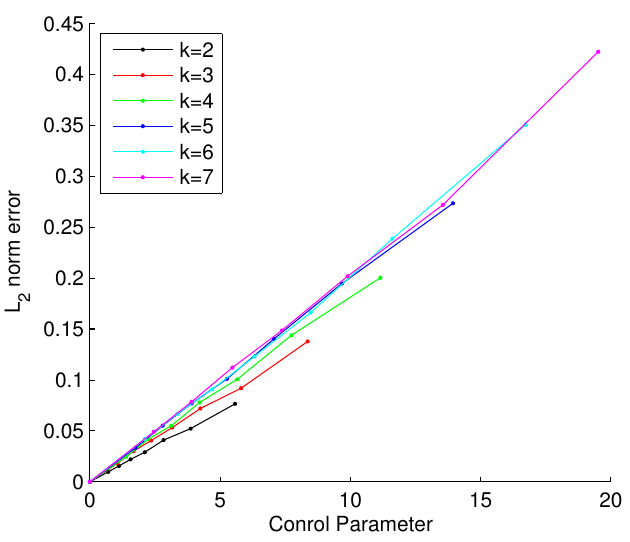} \\
(a) & (b)
\end{tabular}
\caption{\label{fig:mis_low_1} $ \ell_2 $ recovery error for missing noise versus (a) the erasure probability $\rho$ and (b) the control parameter $ \rho\sqrt{k} $. Each point is an average over 200 trials.}
\end{figure}

\subsection{The High-Dimensional Case}

In this subsection, we study the performance of our mod-OMP algorithm for the high-dimensional setting, and compare with the projected gradient method in \cite{loh2011nonconvex}. We first consider the additive noise case and use the following settings: $p=450, n=400,\sigma_{e}=0,\Sigma_{x}=I$, $ \Sigma_w = \sigma_w^2 I $, $ k\in\{2,\ldots,7\} $, and $\sigma_{w} \in [0,1]$. We compare mod-OMP using the $\Sigma_w$-estimator and the projected gradient method using the corresponding $\hat{\Sigma}$ and $\hat{\gamma}$. Figure \ref{fig:add_high} (a) plots the $\ell_{2}$ errors. One observes that OMP outperforms the projected gradient method in all cases. 

\begin{figure}
\hspace{-1cm}
\begin{tabular}{ccc}
\includegraphics[scale=.45]{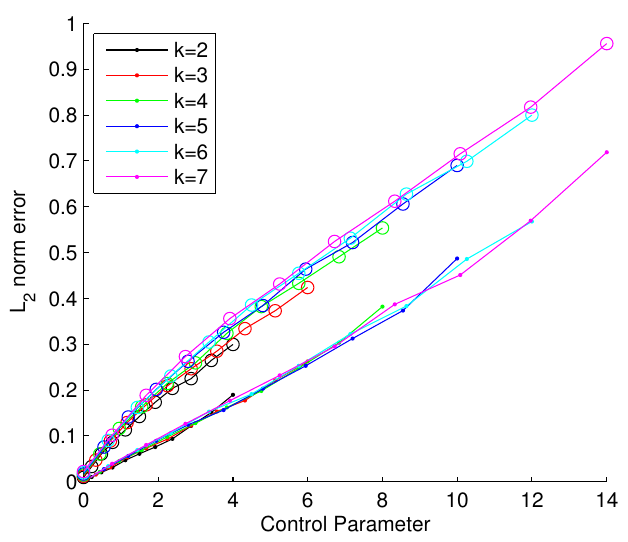} &
\includegraphics[scale=.45]{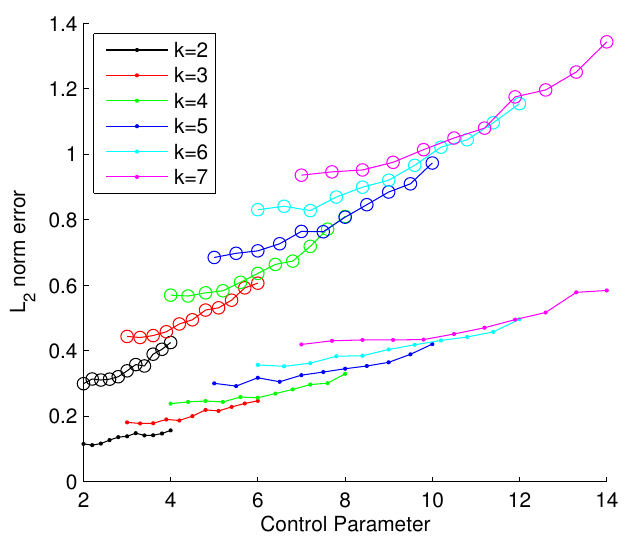} &
\includegraphics[scale=.45]{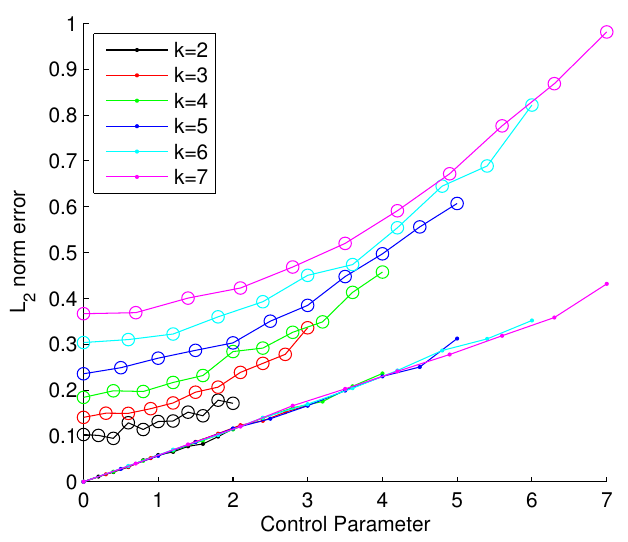} \\
(a) & (b) & (c)
\end{tabular}
\caption{\label{fig:add_high} Comparison of the $ \ell_2 $ recovery error of mod-OMP and the projected gradient method under knowledge of (a) $\Sigma_w$, (b) $\Sigma_x$, and (c) an Instrumental Variable. The error is plotted against the control parameter (a) $ (\sigma_w + \sigma_w^2) k $, (b) $ (1+\sigma_w) k $, and (c) $ \sigma_w  k $. Circles correspond to the projected gradient method and dots to mod-OMP. As claimed, mod-OMP performs better in all cases considered. Each point is an average over 200 trials.}

\end{figure}

We want to point out that mod-OMP enjoys more favorable running time in our experiments, although we do not perform a formal comparison since this depends on the particular implementation of both methods. As is clear from the description of the algorithm, mod-OMP has exactly the same running time as standard OMP. 

We also consider mod-OMP with the $\Sigma_x$- and IV-based estimators. Although not discussed in \cite{loh2011nonconvex}, it is natural to consider the corresponding variants of the projected gradient method which use the $\hat{\Sigma}$ and $\hat{\gamma}$ from knowledge of $\Sigma_x$ or IVs (c.f. (13) in \cite{loh2011nonconvex}). We plot the recovery errors for our two estimators in Figure \ref{fig:add_high} (b) and (c), and again observe better performance of mod-OMP than the projected gradient method.

We further consider robustness of the projected gradient method to over- or under-estimating $\sigma_w$, for support recovery. For very low noise, the performance is unaffected; however, it quickly deteriorates as the noise level grows. The two graphs in Figure \ref{fig:unknown-Sigma} show this deterioration; in contrast, our estimator has excellent support recovery throughout this range.

\begin{figure}
\hspace{-0.8cm}
\begin{tabular}{cc}
\includegraphics[scale=.55]{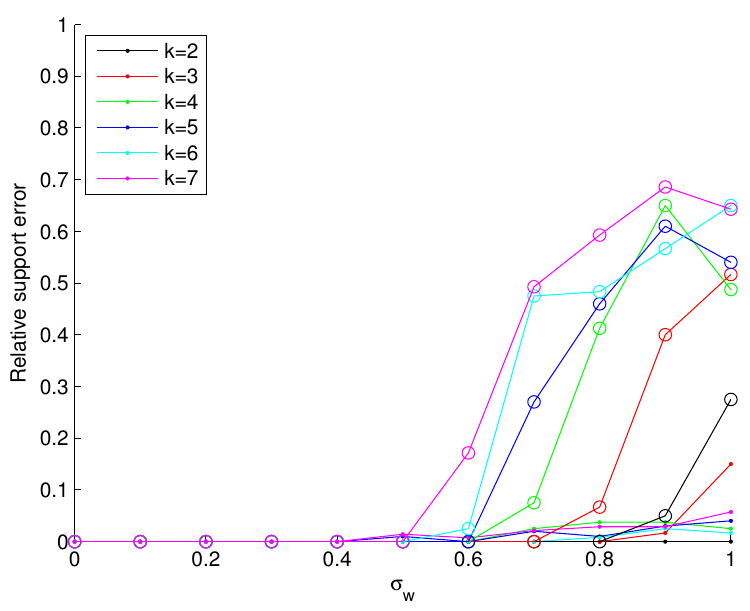} &
\includegraphics[scale=.55]{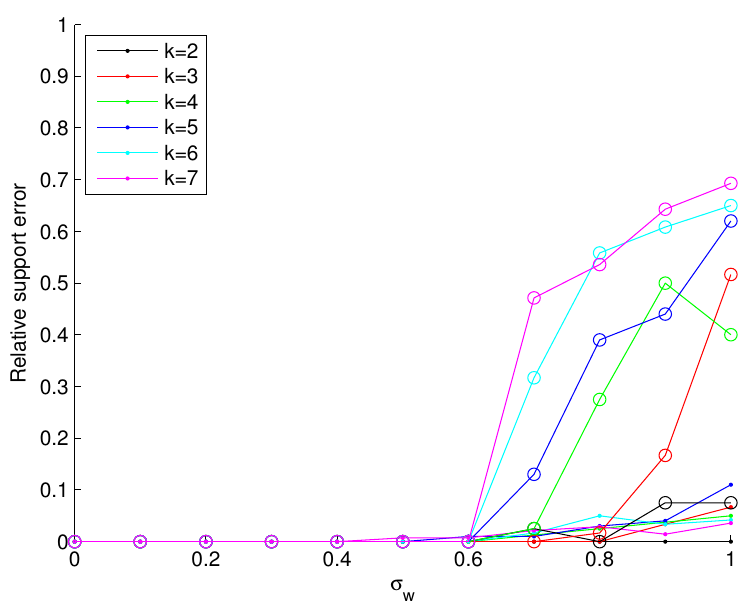} \\
(a) & (b)
\end{tabular}
\caption{\label{fig:unknown-Sigma} Support recovery error for the projected gradient algorithm when $\Sigma_w$ is not precisely known. Figure (a) shows the results when we use $\sigma_w = \sigma_w^{\mbox{{\tiny true}}}/2$, and (b) shows the results for  $\sigma_w = \sigma_w^{\mbox{{\tiny true}}} \times 2$. We note that our approach (which does not require knowledge of $\Sigma_w$) has excellent recovery throughout this range. Each point is an average over 20 trials.}

\end{figure}

We next study the case with missing data with the following setting: $p=750, n=500,\sigma_{e}=0,\Sigma_{x}=I$, $ k\in\{2,\ldots,7\}\} $, and $\rho \in [0,0.5]$. The results are displayed in Figure \ref{fig:mis_high} (a), in which mod-OMP shows better performance.

Finally, although we only consider $X$ with independent columns in this paper, we believe that this restriction can be removed. For now, we corroborate this claim via simulation. Figure \ref{fig:mis_high} (b) shows the results under the following choice of covariance matrix of $X$:
\[
\left(\Sigma_{x}\right)_{ij}=\begin{cases}
1 & \quad i=j\\
0.2 & \quad i\neq j.
\end{cases}
\]
Again, mod-OMP dominates the projected gradient method in terms of empirical performance.

\begin{figure}
\centering{
\begin{tabular}{cc}
\includegraphics[scale=.55]{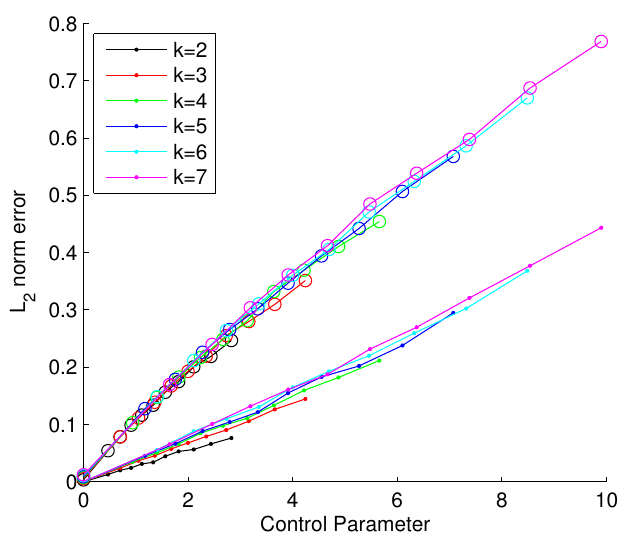} &
\includegraphics[scale=.55]{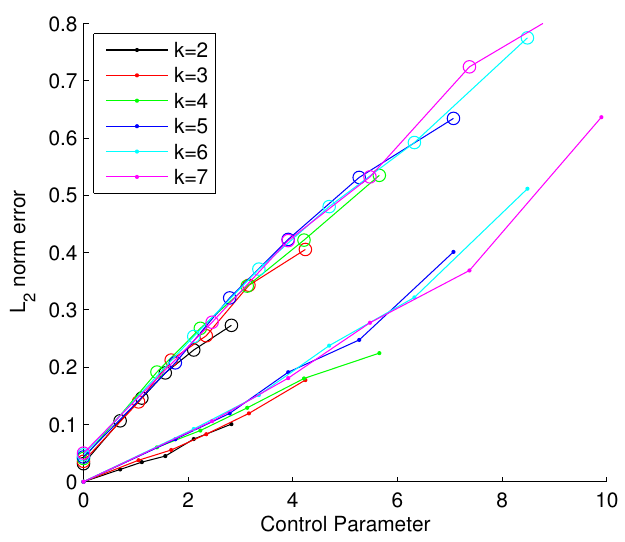} \\
(a) & (b) 
\end{tabular}
\caption{\label{fig:mis_high} Comparison of the $ \ell_2 $ recovery error of mod-OMP and the projected gradient method for missing data. The error is plotted against the control parameter $ k\frac{\sqrt{\rho}}{(1-\rho)} $. (a) Independent columns of $X$, and (b) Correlated columns. Dots correspond to mod-OMP and circles to the projected gradient method. Our results show that mod-OMP performs better in all cases considered. Each point is an average over 50 trials.}
}
\end{figure}

\appendix
\appendixpage

\section{Proof of Supporting Concentration Results}

In this section, we provide the proofs to the concentration results for sub-Gaussian random variables that we make extensive use of in Section \ref{sec:proofs}. We repeat the statements of the results below for convenience.

{\bf Lemma \ref{lem:Sigma_beta}}. {\it 
Suppose $X\in\mathbb{R}^{n\times k}$, $Y\in\mathbb{R}^{n\times m}$
are zero-mean sub-Gaussian matrices with parameters $(\frac{1}{n}\Sigma_{x},\frac{1}{n}\sigma_{x}^{2})$,
$\left(\frac{1}{n}\Sigma_{y},\frac{1}{n}\sigma_{y}^{2}\right)$. Then
for any fixed vector $\mathbf{v}_1,\mathbf{v}_2$, we have
\[
\mathbb{P}\left(\left|\mathbf{v}_1^{\top}\left(Y^{\top}X-\mathbb{E}\left[Y^{\top}X\right]\right)\mathbf{v}_2\right|\ge t\left\Vert \mathbf{v}_1\right\Vert \left\Vert \mathbf{v}_2\right\Vert \right)\le3\exp\left(-cn\min\left\{ \frac{t^{2}}{\sigma_{x}^{2}\sigma_{y}^{2}},\frac{t}{\sigma_{x}\sigma_{y}}\right\} \right).
\]
In particular, if $n\gtrsim\log p\ge\log m,\log k$, we have w.h.p.
\[
\left|\mathbf{v}_1^{\top}\left(Y^{\top}X-\mathbb{E}\left[Y^{\top}X\right]\right)\mathbf{v}_2\right|\le\sigma_{x}\sigma_{y}\left\Vert \mathbf{v}_1\right\Vert \left\Vert \mathbf{v}_2\right\Vert \sqrt{\frac{\log p}{n}}.
\]
Setting $\mathbf{v}_1$ to be the $i^{th}$ standard basis vector, and using a union bound over $i=1,\ldots,m$,
we have w.h.p. \textup{
\[
\left\Vert \left(Y^{\top}X-\mathbb{E}\left[Y^{\top}X\right]\right)v\right\Vert _{\infty}\le\sigma_{x}\sigma_{y}\left\Vert v\right\Vert \sqrt{\frac{\log p}{n}}.
\]
}}

\begin{proof}
Rescaling as necessary, we assume $\sigma_{x}=\sigma_{y}=1$ and $\left\Vert \mathbf{v}_1\right\Vert =\left\Vert \mathbf{v}_2\right\Vert =1$.
Define $\Phi(x)=\left\Vert x\right\Vert ^{2}-\mathbb{E}\left(\left\Vert x\right\Vert ^{2}\right).$
Then $\left|\mathbf{v}_1^{\top}\left(Y^{\top}X-\mathbb{E}\left[Y^{\top}X\right]\right)\mathbf{v}_2\right|=\frac{1}{2}\left|\Phi(X\mathbf{v}_2+Y\mathbf{v}_1)-\Phi(X\mathbf{v}_2)-\Phi(Y\mathbf{v}_1)\right|$.
Note that $X\mathbf{v}_2+Y\mathbf{v}_1=[X,Y][\mathbf{v}_2^{\top},\mathbf{v}_1^{\top}]^{\top}$,
where $X'=[X,Y]$ is zero-mean sub-Gaussian with parameter $(\frac{1}{n}\mathbb{E}\left[X'^{\top}X'\right],\frac{1}{n})$.
Applying (70) in \cite{loh2011nonconvex} to each of the three terms gives
\[
\left|\mathbf{v}_1^{\top}\left(Y^{\top}X-\mathbb{E}\left[Y^{\top}X\right]\right)\mathbf{v}_2\right|\ge t,
\]
with probability at most $\exp\left(-cn\min\left\{ t^{2},t\right\} \right)$.
\end{proof}

{\bf Corollary \ref{cor:X_v_2norm}}. {\it If $X\in\mathbb{R}^{n\times k}$ is a zero-mean
sub-Gaussian matrix with parameter $(\frac{1}{n}\sigma_{x}^{2}I,\frac{1}{n}\sigma_{x}^{2})$,
and $\mathbf{v}$ is a fixed vector in $\mathbb{R}^{n}$, then for any $\epsilon\ge1$,
we have 
\[
\mathbb{P}\left(\left\Vert X^{\top} \mathbf{v} \right\Vert _{2}>\sqrt{\frac{(1+\epsilon)k}{n}}\sigma_{x}\left\Vert \mathbf{v} \right\Vert _{2}\right)\le3\exp\left(-ck\epsilon\right)
\]
}

\begin{proof}
By assmption, $X^{\top}$ is zero-mean sub-Gaussian with parameter
$(\frac{1}{k}\frac{k}{n}\sigma_{x}^{2}I,\frac{1}{k}\frac{k}{n}\sigma_{x}^{2}$)
Note that have $\left\Vert X^{\top} \mathbf{v} \right\Vert _{2}^{2} \le \left| \mathbf{v}^{\top}(XX^{\top}-\frac{k}{n}\sigma_{x}^{2}I) \mathbf{v} \right|+\frac{k}{n}\sigma_{x}^{2}\left\Vert \mathbf{v} \right\Vert _{2}^{2}$.
Applying the last lemma with $t=\frac{k}{n}\sigma_{x}^{2}\epsilon$,
$\epsilon\ge1$ to the first term, we obtain
\[
\mathbb{P}\left(k\left|\mathbf{v}^{\top}(XX^{\top}-\frac{k}{n}\sigma_{x}^{2}I) \mathbf{v} \right|>\frac{k}{n}\sigma_{x}^{2}\epsilon\left\Vert  \mathbf{v} \right\Vert ^{2}\right)\le3\exp\left(-ck\min\left\{ \epsilon^{2},\epsilon\right\} \right)=3\exp\left(-ck\epsilon\right).
\]
The corollary follows.\end{proof}

{\bf Lemma \ref{lem:restricted_e}}. {\it If $X\in\mathbb{R}^{n\times k}$, $Y\in\mathbb{R}^{n\times m}$
are zero mean sub-Gaussian matrices with parameter $(\frac{1}{n}\Sigma_{x},\frac{1}{n}\sigma_{x}^{2})$,$(\frac{1}{n}\Sigma_{y},\frac{1}{n}\sigma_{y}^{2})$,
then 
\[
\mathbb{P}\left(\sup_{\mathbf{v}_1\in\mathbb{R}^{m},\mathbf{v}_2\in\mathbb{R}^{k},\left\Vert \mathbf{v}_1\right\Vert =\left\Vert \mathbf{v}_2\right\Vert =1}\left|\mathbf{v}_1^{\top}\left(Y^{\top}X-\mathbb{E}\left[Y^{\top}X\right]\right)\mathbf{v}_2\right|\ge t\right)\le2\exp\left(-cn\min(\frac{t^{2}}{\sigma_{x}^{2}\sigma_{y}^{2}},\frac{t}{\sigma_{x}\sigma_{y}})+6(k+m)\right).
\]
In particular, for each $\lambda>0$, if $n\gtrsim\max\left\{ \frac{\sigma_{x}^{2}\sigma_{y}^{2}}{\lambda^{2}},1\right\} (k+m)\log p$,
then w.h.p.
\begin{eqnarray*}
\sup_{\mathbf{v}_1\in\mathbb{R}^{m},\mathbf{v}_2\in\mathbb{R}^{k}}\left|\mathbf{v}_1^{\top}\left(Y^{\top}X-\mathbb{E}\left[Y^{\top}X\right]\right)\mathbf{v}_2\right| & \le & \frac{1}{54}\lambda\left\Vert \mathbf{v}_1\right\Vert \left\Vert \mathbf{v}_2\right\Vert 
\end{eqnarray*}
}

\begin{proof}
Rescaling as necessary, we assume $\sigma_{x}=\sigma_{y}=1$. Let $\mathcal{A}_{1}$, be a $1/3$-cover of $\mathbf{v}_1=\{ \mathbf{v} \in\mathbb{R}^{m},\left\Vert v\right\Vert \le1\}$;
it is known that $\left|\mathcal{A}_{1}\right|\le9^{2m}$, and for
each $\mathbf{v}$, there is a $u(\mathbf{v})\in\mathcal{A}_{1}$ such that $\left\Vert \Delta(\mathbf{v})\right\Vert \triangleq\left\Vert \mathbf{v}-u(\mathbf{v})\right\Vert \le\frac{1}{3}$.
Similarly we can find a $1/3$-cover $\mathcal{A}_{2}$ of $\mathbf{v}_2=\{\mathbf{v} \in\mathbb{R}^{k},\left\Vert \mathbf{v} \right\Vert \le1\}$
with $\left|\mathcal{A}_{2}\right|\le9^{2k}$. Defining $\Phi(\mathbf{v}_1,\mathbf{v}_2)=\mathbf{v}_1^{\top}\left(Y^{\top}X-\mathbb{E}\left[Y^{\top}X\right]\right)\mathbf{v}_2$, then
\begin{eqnarray*}
\sup_{\mathbf{v}_1\in \mathbf{v}_1,\mathbf{v}_2\in \mathbf{v}_2}\left|\Phi(\mathbf{v}_1,\mathbf{v}_2)\right| & \le & \max_{u_{1}\in\mathcal{A}_{1},u_{2}\in\mathcal{A}_{2}}\left|\Phi(u_{1},u_{2})\right| + \sup_{\mathbf{v}_1\in \mathbf{v}_1,\mathbf{v}_2\in \mathbf{v}_2}\left|\Phi(\Delta(\mathbf{v}_1),u(\mathbf{v}_2))\right| \\
&& + \sup_{\mathbf{v}_1\in \mathbf{v}_1,\mathbf{v}_2\in \mathbf{v}_2}\left|\Phi(u(\mathbf{v}_1),\Delta(\mathbf{v}_2))\right|+\sup_{\mathbf{v}_1\in \mathbf{v}_1,\mathbf{v}_2\in \mathbf{v}_2}\left|\Phi(\Delta(\mathbf{v}_1),\Delta(\mathbf{v}_2))\right|.
\end{eqnarray*}
Becaue $3\Delta(\mathbf{v}_1),u(\mathbf{v}_1)\in \mathbf{v}_1$, and $3\Delta(\mathbf{v}_2),u(\mathbf{v}_2)\in \mathbf{v}_2$,
it follows that 
\begin{eqnarray*}
\sup_{\mathbf{v}_1\in \mathbf{v}_1,\mathbf{v}_2\in \mathbf{v}_2}\left|\Phi(\mathbf{v}_1,\mathbf{v}_2)\right| & \le & \max_{u_{1},u_{2}\in\mathcal{A}}\left|\Phi(u_{1},u_{2})\right|+\left(\frac{1}{3}+\frac{1}{3}+\frac{1}{9}\right)\sup_{\mathbf{v}_1\in \mathbf{v}_1,\mathbf{v}_2\in \mathbf{v}_2}\left|\Phi(\mathbf{v}_1,\mathbf{v}_2)\right|,
\end{eqnarray*}
hence $\sup_{\mathbf{v}_1\in \mathbf{v}_1,\mathbf{v}_2\in \mathbf{v}_2}\left|\Phi(\mathbf{v}_1,\mathbf{v}_2)\right|\le\frac{9}{2}\max_{u_{1},u_{2}\in\mathcal{A}}\left|\Phi(u_{1},u_{2})\right|.$
Using the last lemma and a union bound, we obtain
\[
\mathbb{P}\left(\frac{9}{2}\max_{u_{1},u_{2}\in\mathcal{A}}\left|\Phi(u_{1},u_{2})\right|\ge t\right)\le9^{2k+2m}\cdot\exp\left(-cn\min\left\{ t^{2},t\right\} \right)\le\exp\left(-cn\min\left\{ t^{2},t\right\} +6(k+m)\right).
\]
 
\end{proof}

\bibliographystyle{plain}
\bibliography{ompeivbib}

\end{document}